\documentclass[12pt]{amsart}
\usepackage{latexsym}
\usepackage{amssymb}
\usepackage{amscd}
\usepackage{mathrsfs}

\renewcommand\a{\alpha}
\renewcommand\b{\beta}
\newcommand\g{\gamma}
\renewcommand\d{\delta}
\newcommand\la{\lambda}

\renewcommand\th{\theta}

\newcommand\s{\sigma}

\renewcommand\r{\rho}

\newcommand\w{\omega}

\newcommand\vD{\varDelta}

\newcommand\ve{\varepsilon}

\newcommand\BA{\mathbf A}

\newcommand\BQ{\mathbf Q}
\newcommand\BF{\mathbf F}
\newcommand\BC{\mathbf C}
\newcommand\BB{\mathbf B}

\newcommand\BZ{\mathbf Z}

\newcommand\BN{\mathbf N}

\newcommand\bB{\mathbf B}

\newcommand\BU{\mathbf U}
\newcommand\BV{\mathbf V}

\newcommand\Bb{\mathbf b}
\newcommand\Bc{\mathbf c}

\newcommand\Bh{\mathbf h}

\newcommand\Bd{\mathbf d}

\newcommand\SC{\mathscr{C}}

\newcommand\SL{\mathscr{L}}

\newcommand\ST{\mathscr{T}}
\newcommand\SX{\mathscr{X}}

\newcommand\Fg{\mathfrak g}

\newcommand\Fh{\mathfrak h}

\newcommand\iv{^{-1}}
\newcommand\wh{\widehat}
\newcommand\wt{\widetilde}

\newcommand\ol{\overline}
\newcommand\ul{\underline}

\newcommand\lra{\leftrightarrow}

\newcommand\lp{\operatorname{\!\langle\!}}
\newcommand\rp{\operatorname{\!\rangle\!}}

\newcommand\cl{\operatorname{cl}}

\newcommand\re{\operatorname{re}}
\newcommand\ima{\operatorname{im}}

\newcommand{\isom}{\,\raise2pt\hbox{$\underrightarrow{\sim}$}\,}
\numberwithin{equation}{section}

\newtheorem{thm}{Theorem}[section]
\newtheorem{lem}[thm]{Lemma}
\newtheorem{cor}[thm]{Corollary}
\newtheorem{prop}[thm]{Proposition}

\def \para#1{\par\medskip\textbf{#1}
              \addtocounter{thm}{1}}

\def \remark#1{\par\medskip\noindent
                \textbf{Remark #1}
                \addtocounter{thm}{1}}

\begin{document}
\setlength{\baselineskip}{4.9mm}
\setlength{\abovedisplayskip}{4.5mm}
\setlength{\belowdisplayskip}{4.5mm}
\renewcommand{\theenumi}{\roman{enumi}}
\renewcommand{\labelenumi}{(\theenumi)}
\renewcommand{\thefootnote}{\fnsymbol{footnote}}
\renewcommand{\thefootnote}{\fnsymbol{footnote}}
\allowdisplaybreaks[2]
\parindent=20pt
\medskip
\begin{center}
  {\bf Diagram automorphisms and quantum groups } 
\par
\vspace{1cm}
Toshiaki Shoji and Zhiping Zhou
\\
\vspace{0.7cm}
\title{}
\end{center}

\begin{abstract}
Let $\BU^-_q = \BU^-_q(\Fg)$ be the negative part of the quantum group associated to
a finite dimensional simple Lie algebra $\Fg$, and $\s : \Fg \to \Fg$
be the automorphism obtained from the diagram automorphism. 
Let $\Fg^{\s}$ be the fixed point subalgebra of $\Fg$, and put 
$\ul\BU^-_q = \BU^-_q(\Fg^{\s})$. Let $\BB$ be the canonical basis of $\BU_q^-$
and $\ul\BB$ the canonical basis of $\ul\BU_q^-$. $\s$ induces a natural action 
on $\BB$, and we denote by $\BB^{\s}$ the set of $\s$-fixed elements in $\BB$. 
Lusztig proved that there exists a canonical bijection $\BB^{\s} \simeq \ul\BB$
by using geometric considerations.  In this paper, we construct such a bijection
in an elementary way.
We also consider such a bijection in the case of certain affine quantum groups, 
by making use of PBW-bases constructed by Beck and Nakajima. 
\end{abstract}

\let\thefootnote\relax\
\maketitle
\markboth{SHOJI-ZHOU}{Diagram automorphisms}
\pagestyle{myheadings}


\begin{center}
{\sc Introduction }
\end{center}

\para{0.1.}
Let $X$ be a Dynkin diagram with vertex set $I$, and $\Fg$ the semisimple Lie algebra
associated to $X$.  We denote by $\BU_q = \BU_q(\Fg)$ the quantum  enveloping algebra 
of $\Fg$, and by $\BU_q^-$ its negative part, which are associative algebras over $\BQ(q)$.  
Let $W$ be the Weyl group of $\Fg$, and $w_0$  the longest element of $W$.
Let $\Bh = (i_1, \dots, i_{\nu})$ be a sequence of $i_k \in I$ such that 
$w_0 = s_{i_1}\cdots s_{i_{\nu}}$ gives a reduced expression of $w_0$, where $s_i (i \in I)$
are simple reflections in $W$. 
For each $\Bh$ as above, there exists a basis $\SX_{\Bh}$ of $\BU^-_q$, called the PBW-basis
of $\BU_q^-$. Put $\BA = \BZ[q, q\iv]$, and let ${}_{\BA}\BU_q^-$ be Lusztig's integral 
form of $\BU_q^-$. 

We consider the following statements. 
\par\medskip
\noindent
(0.1.1) 
\begin{enumerate}
\item  
The $\BZ[q]$-submodule of $\BU_q^-$ generated by $\SX_{\Bh}$ is independent 
of the choice of $\Bh$, which we denote by $\SL_{\BZ}(\infty)$. 
\item 
The $\BZ$-basis of $\SL_{\BZ}(\infty)/q\SL_{\BZ}(\infty)$ induced from $\SX_{\Bh}$ 
is independent of the choice of $\Bh$.
\item 
For each $\Bh$, PBW-basis $\SX_{\Bh}$ gives rise to an $\BA$-basis of ${}_{\BA}\BU_q^-$. 
\end{enumerate}
We also consider a weaker version of (iii), 
\par\medskip
(iii$\,'$) \ For each $\Bh$, any element of $\SX_{\Bh}$ is contained in 
${}_{\BA}\BU_q^-$. 
\par\medskip
 
The canonical basis $\BB$ of $\BU^-_q$ was constructed by Lusztig [L2, L3] by using 
a geometric method. It is known that it coincides with 
the global crystal basis of Kashiwara [K1].  
\par 
The statement (0.1.1) can be verified in general 
by making use of the canonical basis or Kashiwara's global crystal basis.

\para{0.2.} 
We are interested in an elementary construction of canonical bases, in the sense 
that we don't appeal to Lusztig's geometric theory of canonical bases nor Kashiwara's 
theory of crystal bases. We shall construct canonical bases (as discussed in [L2]), 
by making use of PBW-basis, based on the properties (0.1.1). 
Actually, in their  theories, canonical bases or crystal bases 
are constructed independently from PBW-bases. However those constructions look 
like a huge black box, and it is not easy to trace the construction 
even in the small rank cases. On the other hand, the construction of 
PBW-bases is more explicit, the parametrization is easy, and   
they fit to direct computations.  
So it is important to express canonical basis in terms of PBW-bases, which is 
the problem closely
related to the elementary construction of canonical basis.       
\par 
In the case where $X$ is simply laced, the verification of (0.1.1) is rather easy. 
In the non-simply laced case, the problem is reduced to the case of type $B_2$ or 
$G_2$.  In the case of $B_2$, 
the properties (i) and (iii) were verified by [L1], by computing the commutation
relations of root vectors in the case of type $B_2$, and furthermore by applying the 
method of Kostant on the $\BZ$-form of Chevalley groups in the case of type $G_2$.
Later [X1] gave a proof 
of (iii) similar to the case of $B_2$. But in any case, it requires a hard 
computation.  
In [X2], Xi computed, in the case of $B_2$,  the canonical basis of $\BU^-_q$  
explicitly in terms of PBW-basis.  The property (ii) follows from his result. 
But the property (ii) for $G_2$ is not yet verified (in an elementary method).  
\par
If we assume (i) and (iii) in (0.1.1), one can construct the ``canonical basis'', 
which is only independent of $\Bh$, up to $\pm 1$. We call them  
the signed basis of $\BU^-_q$.  Thus in the non-simply laced case, one can construct
the signed basis.

\para{0.3.}
Assume that $X$ is simply laced, 
and let $\s$ be a graph automorphism of $X$. We denote by $\ul I$ the set of 
orbits in $I$ under the action of $\s : I \to I$.  Then $\s$ determines a Dynkin diagram 
$\ul X$ whose vertex set is given by $\ul I$. 
$\ul X$ corresponds to the $\s$-fixed point subalgebra $\Fg^{\s}$ of $\Fg$, 
and we denote by $\ul\BU_q = \BU_q(\Fg^{\s})$ the corresponding  
quantum  enveloping algebra, and $\ul\BU_q^-$ its negative part.  
Let $\BB$ be the canonical basis of $\BU^-_q$.  Then $\s$ permutes $\BB$, and 
we denote by $\BB^{\s}$ the set of $\s$-fixed elements in $\BB$.
We also denote by $\ul\BB$ the set of canonical basis of $\ul\BU^-_q$. 
In [L4] (and in [L3]), Lusztig proved that there exists a canonical bijection 
between $\BB^{\s}$ and $\ul\BB$, based on geometric considerations of canonical 
basis. 
\par
In this paper, we construct the bijection $\BB^{\s} \isom \ul \BB$ in an elementary 
way. 
We assume that $\s$ is admissible, namely for $\eta \in \ul I$, 
if $i, j \in \eta$ with $i \ne j$, then $i$ and $j$ are not joined in $X$. 
Let $\ve$ be the order of $\s$.  
We assume that 
$\ve = 2$ or 3 (note that if $X$ is irreducible, then 
$\ve = 2$ or 3). . 
Let $\BF$ be the finite field $\BZ/\ve\BZ$, and 
put $\BA' = \BF[q,q\iv] = \BA/\ve\BA$.  
Let ${}_{\BA}\BU^{-,\s}_q$ be the subalgebra of ${}_{\BA}\BU^-_q$ consisting of 
$\s$-fixed elements, and consider the $\BA'$-algebra
${}_{\BA'}\BU^{-,\s}_q = {}_{\BA}\BU^{-,\s}_q\otimes_{\BA}\BA'$. 
Let $J$ be the $\BA'$-submodule of ${}_{\BA'}\BU^{-,\s}_q$ consisting of elements 
of the form $\sum_{0 \le i < \ve}\s^i(x)$ for $x \in {}_{\BA'}\BU^-_q$.  
Then $J$ is a two-sided ideal of ${}_{\BA'}\BU^{-,\s}_q$, and we denote by 
$\BV_q$ the quotient algebra ${}_{\BA'}\BU^{-,\s}_q/J$. 
We define ${}_{\BA'}\ul\BU_q^-$ similarly to ${}_{\BA'}\BU_q^-$. 
We can prove the following result (Proposition 1.20 and Corollary 1.21).
\begin{thm}   
Assume that {\rm (iii)} in {\rm (0.1.1)} holds for ${}_{\BA}\BU_q^-$, 
and {\rm (iii$\, '$)}  holds for ${}_{\BA}\ul\BU_q^-$.  Then we have an isomorphism  of 
$\BA'$-algebras
\begin{equation*}
\tag{0.4.1}
{}_{\BA'}\ul\BU_q^- \simeq \BV_q.
\end{equation*} 
Moreover {\rm (iii)} holds for ${}_{\BA}\ul\BU_q^-$. 
\end{thm} 

By Theorem 0.4, one can define the signed basis for $\ul\BU^-_q$ by assuming (iii$\,'$).
But in the case of $G_2$, we have a more precise result (Proposition 1.23), namely 

\begin{prop} 
Let $\ul\BU_q^-$ be of type $G_2$. Then the ambiguity of the sign can be removed in 
the signed basis, hence {\rm (ii)} of {\rm (0.1.1)} holds for $\ul\BU_q^-$. 
\end{prop}

(0.4.1) gives a surjective map ${}_{\BA'}\BU_q^{-,\s} \to {}_{\BA'}\ul\BU_q^-$ 
combined with the natural surjection ${}_{\BA'}\BU_q^{-,\s} \to \BV_q$.  
This map is compatible with PBW-bases, hence induces a natural map 
$\BB^{\s} \to \ul\BB$, which is shown to be bijective (see Remark 1.24). 
Thus we can recover Lusztig's bijection $\bB^{\s} \isom \ul\bB$ by 
an elementary method. 

\para{0.6.}
In Beck and Nakajima [BN], PBW-bases were constructed for 
the affine quantum enveloping algebras $\BU_q^-$.  They showed that 
an analogous property of (iii$'$) holds for those PBW-basis, and 
that of (iii) holds if the corresponding diagram $X$ is simply laced.  
We apply the previous discussion 
to the case where $X$ is simply laced of type $A_{2n+1}^{(1)}$ ($n \ge 1$), 
$D_n^{(1)}$ ($n \ge 4$), $E_6^{(1)}$ with $\ve = 2$, and $D_4^{(1)}$ with 
$\ve = 3$.  Then $\ul X$ is twisted affine 
of type $D^{(2)}_{n+2}$, $A^{(2)}_{2n-3}$, $E_6^{(2)}$ and $D_4^{(3)}$, 
respectively (under the notation in [Ka, 4.8]). We have (Corollary 2.17),

\begin{thm}  
Assume that $\ul X$ is twisted of type $D^{(2)}_n$, $A^{(2)}_{2n-1}$, $E^{(2)}_6$ 
or $D^{(3)}_4$.   Then {\rm (iii)} holds for $\ul\BU_q^-$.  
Moreover the surjective map ${}_{\BA'}\BU_q^{-,\s} \to {}_{\BA'}\ul\BU_q^-$
gives a natural bijection $\BB^{\s} \isom \ul\BB$. 
\end{thm}  

\remark{0.8.}
Assume that $\Fg$ is an affine Lie algebra, and $\Fg_0$ the associated 
finite dimensional subalgebra of $\Fg$. We consider the automorphism 
$\s : \Fg \to \Fg$. In order to apply the construction of PBW-basis in [BN]
to our $\s$-setting, we need to assume that $\s$ leaves $\Fg_0$ invariant. 
Then $\Fg^{\s}$ is necessarily twisted affine type.  Our discussion can 
not cover the case where $\Fg^{\s}$ is untwisted type.  

\para{0.9.}
As mentioned in 0.3, Lusztig has given a canonical bijection 
between the set of $\s$-stable canonical bases of $\BU_q^-$ and 
the set of canonical bases of $\ul\BU_q^-$. A closely related problem 
for crystal bases  was also studied by many researchers, such as  
Naito and Sagaki [NS], Savage [S].  However those results are concerned 
with the level of the parametrization, since there exists no
direct relationship between $\BU_q^{\s}$ and $\ul\BU_q^-$.     
The main observation in our work is that if we replace $\BA = \BZ[q,q\iv]$
by $\BA' = (\BZ/\ve\BZ)[q,q\iv]$, we obtain a natural surjective map 
from ${}_{\BA'}\BU_q^{-,\s}$ to ${}_{\BA'}\ul\BU_q^-$ as $\BA'$-algebras.
This has an advantage that we can compare directly the algebra structure of 
$\ul\BU_q^-$ and of $\BU_q^{-,\s}$, not only the correspondence of bases.  
For example, the following is an easy consequence of our results. 
(Notations are as in Section 1 for the Dynkin case. A similar result 
also holds for the affine case.)
\begin{thm}  
Let $b(\ul\Bc, \ul\Bh)$ be a canonical basis of $\ul\BU_q^-$, and 
$b(\Bc, \Bh)$ the corresponding $\s$-stable canonical basis of $\BU_q^-$.
We write them as 
\begin{align*}
b(\ul\Bc, \ul\Bh) &= 
   L(\ul\Bc, \ul\Bh) + \sum_{\ul\Bd > \ul \Bc}a_{\ul\Bd}L(\ul\Bd, \ul\Bh),  \\
b(\Bc, \Bh) &= L(\Bc, \Bh) + \sum_{\Bd' > \Bc}a'_{\Bd'}L(\Bd', \Bh), 
\end{align*}
with $a'_{\Bd'}, a_{\ul\Bd} \in q\BZ[q]$. 
If $L(\Bd', \Bh)$ is the $\s$-stable PBW-basis 
corresponding to $L(\ul\Bd, \ul\Bh)$, then we have 
$a'_{\Bd'} \equiv a_{\ul\Bd} \pmod \ve$.  
\end{thm}

Some examples of Theorem 0.10 for small rank cases were 
computed in [MNZ]. 

\par\medskip
This research has grown up from the question,  
concerning the elementary construction of canonical bases, posed by 
H. Nakajima in his lecture note [N] on the lectures at Sophia University, 2006. 
The authors are grateful to him for his helpful suggestions.  

\par\bigskip
\section{PBW-bases and canonical bases}

\para{1.1.}
In this paper, we understand that a Cartan datum is a pair 
$X = (I, (\ , \ ))$, where 
$(\ ,\ )$ is a symmetric bilinear form 
on $\bigoplus_{i \in I}\BQ\a_i$ (a finite dimensional vector space 
over $\BQ$ with the basis $\{ \a_i\}$ indexed by $I$) such that 
$(\a_i, \a_j) \in \BZ$, satisfying the property
\par
\begin{itemize}
\item
$(\a_i,\a_i) \in 2\BZ_{> 0}$ for any $i \in I$, 
\item
$\frac{2(\a_i,\a_j)}{(\a_i,\a_i)} \in \BZ_{\le 0}$ for any $i \ne j$ in $I$.
\end{itemize}

The Cartan datum $X$ is called simply laced if $(\a_i, \a_j) \in \{ 0,-1\}$ for 
any $i \ne j$ in $I$, and $(\a_i,\a_i) = 2$ for any $i \in I$.   
The Cartan datum  $X$ determines a graph with the vertex set $I$.  
If the associated graph is connected, $X$ is said to be irreducible. 
Put $a_{ij} = 2(\a_i, \a_j)/(\a_i,\a_i)$ for any $i,j \in I$. 
The matrix $(a_{ij})$ is called the Cartan matrix.  
\par
In the case where the bilinear form is positive definite, $X$ is called finite type.
In that case, the associated graph is a Dynkin diagram.  In the case where 
the bilinear form is positive semi-definite, $X$ is called affine type.  In that case, 
the associated graph is a Euclidean diagram.  In this paper, we are concerned with 
$X$ of finite type or affine type. 

\para{1.2.} 
Let $X = (I, (\ ,\ ))$ be a simply laced Cartan datum, and 
let $\s : I \to I$ be a permutation such that 
$(\s(\a_i), \s(\a_j)) = (\a_i,\a_j)$ for any $i, j \in I$.   
Let $\ul I$ be the set of orbits of $\s$ on $I$. 
We assume that $\s$ is admissible, namely for each orbit $\eta \in \ul I$, 
$(\a_i, \a_j) = 0$ for any $i \ne j$ in $\eta$.  
\par
We define a symmetric bilinear form $(\ ,\ )_1$ on 
$\bigoplus_{\eta \in \ul I}\BQ \a_{\eta}$ by 
\begin{equation*}
(\a_{\eta}, \a_{\eta'})_1 = \begin{cases}
                               2|\eta|  &\quad\text{ if } \eta = \eta', \\
               - |\{ (i, j) \in \eta \times \eta' \mid (\a_i, \a_j) \ne 0\}|
                           &\quad\text{ if } \eta \ne \eta'.
                             \end{cases}                               
\end{equation*}
It is easy to see that $\ul X = (\ul I, (\ ,\ )_1)$ defines a Cartan datum. 

\para{1.3.}
Let $I = \{1,2, \dots, 2n -1\}$ for $n \ge 1$.  For $i, j \in I$, 
we put $(\a_i, \a_j) = 2$ if $i = j$, $(\a_i,\a_j) = -1$ if 
$i - j = \pm 1$, and $(\a_i, \a_j) = 0$ otherwise. 
Then $(I, (\ ,\ ))$ is a simply laced irreducible Cartan datum of type $A_{2n -1}$.
We define a permutation $\s : I \to I$ by $\s(i) = 2n -i$ for all $i$. 
Then $\s$ satisfies the condition in 1.2.  We can identify $\ul I$ with 
the set $\{ \ul 1, \dots, \ul n \}$, where $\ul i = \{ i, 2n - i\}$ for 
$1 \le i \le n-1$ and $\ul n = \{ n \}$. 
Then $(\ul I, (\ ,\ )_1)$ is the Cartan datum of type $B_n$. 

\para{1.4.}
Let $I = \{ 1, 2, 2',2''\}$.  
We define a permutation 
$\s : I \to I$ of order 3 
by $\s(1) = 1$ and $\s  : 2 \mapsto 2' \mapsto 2'' \mapsto 2$.
The set $\ul I$ of orbits of $\s$ in $I$ is given by 
$\ul I = \{ \ul 1, \ul 2\}$, where $\ul 1 = \{ 1 \}$ and 
$\ul 2 = \{ 2, 2', 2'' \}$. We define a symmetric bilinear form 
on $\bigoplus_{i \in I}\BQ\a_i$ by
\begin{equation*}
(\a_i, \a_j) = \begin{cases}
                   2  &\quad\text{ if } i = j, \\
                   -1 &\quad\text{ if $i \in \ul 1, j \in \ul 2$
                             or $i \in \ul 2, j \in \ul 1$}, \\
                   0  &\quad\text{ if } i, j \in \ul 2, i \ne j.
               \end{cases} 
\end{equation*}
Then $(I, (\ ,\ ))$ gives the Cartan datum of type $D_4$. 
$\s : I \to I$ satisfies the condition in 1.2, and $(\ul I, (\ ,\ )_1)$
gives the Cartan datum of type $G_2$. 

\para{1.5.}
Let $q$ be an indeterminate, and for an integer $n$, 
a positive integer $m $, put 
\begin{equation*}
[n]_q = \frac{q^n - q^{-n}}{q - q\iv}, \quad [m]_q^! = \prod_{i = 1}^m[i]_q, 
         \quad [0]!_q = 1. 
\end{equation*} 
For each $i \in I$, put $q_i = q^{(\a_i,\a_i)/2}$, and consider $[n]_{q_i}$, etc.  
by replacing $q$ by $q_i$ in the above formulas. 
Let $\BU_q^-$ be the negative part of the quantum enveloping algebra $\BU_q$ 
associated to a Cartan datum $X = (I, (\ ,\ ))$.  Hence $\BU_q^-$ is 
an associative algebra over $\BQ(q)$ with generators $f_i$ ($i \in I$)
satisfying the fundamental relations 
\begin{equation*}
\tag{1.5.1}
\sum_{k = 0}^{1-a_{ij}}(-1)^kf_i^{(k)}f_jf_i^{(1- a_{ij} - k)} = 0
\end{equation*}
for any $i \ne j \in I$, where $f_i^{(n)} = f_i^n/[n]^!_{q_i}$ for 
a non-negative integer $n$. 
\par
We now assume that the Cartan datum $X$ is simply-laced.
Then $[n]_{q_i} = [n]_q$ for any $i \in I$. 
Let $\s : I \to I$ be the automorphism as in 1.2.  
Then $\s$ induces an algebra automorphism $\s : \BU_q^- \isom \BU_q^-$ 
by $f_i \to f_{\s(i)}$. 
We denote by $\BU_q^{-,\s}$ the subalgebra of $\BU_q^-$ consisting of $\s$-fixed 
elements. 
Let $\BA = \BZ[q,q\iv]$, and ${}_{\BA}\!\BU_q^-$ be the $\BA$-subalgebra of 
$\BU_q^-$ generated by $f_i^{(a)}$ for $i \in I$ and $a \in \BN$. 
Then $\s$ stabilizes ${}_{\BA}\!\BU_q^-$, and we can define ${}_{\BA}\!\BU_q^{-,\s}$ 
the subalgebra of ${}_{\BA}\!\BU_q^-$ consisting of $\s$-fixed elements. 
\par
Let $\ul X = (\ul I, (\ ,\ )_1)$ be the Cartan datum obtained from $\s$ as in 1.2.
We denote by $\ul{\BU}_q^-$ the negative part of the quantum enveloping algebra associated to
$\ul X$, namely, $\ul{\BU}_q^-$ is the $\BQ(q)$-algebra generated by 
$\ul f_{\eta}$ with $\eta \in \ul I$ satisfying a similar relation as in (1.5.1).
\par
Let $\ve$ be the order of $\s$ (here we assume tht $\ve = 2$ or 3), and let 
$\BF = \BZ/\ve\BZ$ be the finite field of $\ve$-elements. 
Put $\BA' = \BF[q,q\iv]$, and consider the $\BA'$-algebra
\begin{equation*}
\tag{1.5.2}
{}_{\BA'}\BU_q^{-,\s} = {}_{\BA}\BU_q^{-,\s}\otimes_{\BA}\BA'
                      \simeq {}_{\BA}\BU_q^{-,\s}/\ve({}_{\BA}\BU_q^{-,\s}).
\end{equation*}    
Let $J$ be the $\BA'$-submodule of ${}_{\BA'}\BU_q^{-,\s}$ consisting of 
elements of the form $\sum_{0 \le i < \ve}\s^i(x)$ for $x \in {}_{\BA'}\BU_q^-$.  
Then $J$ is a two-sided ideal of ${}_{\BA'}\BU_q^{-,\s}$, and we denote by 
$\BV_q$ the quotient algebra ${}_{\BA'}\BU_q^{-,\s}/J$. 
Let $\pi : {}_{\BA'}\BU_q^{-,\s} \to \BV_q$ be the natural map.  
\par
Let $\ul\BU_q^-$ be as before. 
We can define ${}_{\BA}\ul\BU_q^-$ and ${}_{\BA'}\ul\BU_q^-$ similarly to 
${}_{\BA}\BU_q^-$ and ${}_{\BA'}\BU_q^-$. 

\para{1.6.}
In the rest of this section, we assume that $X$ is of finite type. 
Let $W$ be the Weyl group associated to the Cartan datum $X$, with 
simple reflections $\{ s_i \mid i \in I\}$. Let $l : W \to \BN$ 
be the standard length function of $W$ relative to the generators 
$s_i \ (i \in I)$. Let $w_0$ be the unique longest element in $W$ with 
respect to $l$, and put $\nu = l(w_0)$. 
Let $\ul W$ be the Weyl group associated to the Cartan datum $\ul X$, with 
simple reflections $\{ \ul s_{\eta} \mid \eta \in \ul I\}$. 
Then $\ul l, \ul w_0, \ul \nu$ with respect to $\ul X$ 
are defined similarly to $l, w_0, \nu$.  
For any $\eta \in \ul I$, let $w_{\eta}$ be the product of 
$s_i$ for $i \in \eta$ (note, by our assumption, that such $s_i$ are mutually 
commuting).  Then $\ul W$ can be identified with the subgroup of $W$ 
generated by $\{ w_{\eta} \mid \eta \in \ul I\}$ under the correspondence 
$\ul s_{\eta} \lra w_{\eta}$. 
The map $s_i \mapsto s_{\s(i)}$ defines an automorphism $\s : W \to W$, 
and $\ul W$ coincides with the subgroup $W^{\s} = \{ w \in W \mid \s(w) = w\}$ of 
$W$ under the above identification. 
We have $w_0 = \ul w_0$, and if $\ul w_0 = \ul s_{\eta_1}\cdots \ul s_{\eta_{\ul \nu}}$
is a reduced expression of $\ul w_0$,
then $w_0 = w_{\eta_1}\cdots w_{\eta_{\ul \nu}}$, which  satisfies the relation 
$\sum_{k = 1}^{\ul \nu} l(w_{\eta_k}) = \nu$.  
Thus if we write $w_{\eta} = \prod_{i \in \eta}s_i$ for any $\eta \in \ul I$, 
$w_0 = w_{\eta_1}\cdots w_{\eta_{\ul \nu}}$ induces a reduced expression of $w_0$, 
\begin{equation*}
\tag{1.6.1}
w_0 =  
  (\prod_{k_1 \in \eta_1}s_{k_1})\cdots (\prod_{k_{\ul \nu} \in \eta_{\ul \nu}}s_{k_{\ul \nu}})
     = s_{i_1}\cdots s_{i_{\nu}}.
\end{equation*}
We write $\ul \Bh = (\eta_1, \dots, \eta_{\ul \nu})$ and 
$\Bh = (i_1, \dots, i_{\nu})$. Note that $\Bh$ is determined from $\ul \Bh$ by choosing 
the expression $w_{\eta} = s_{k_1}\cdots s_{k_{|\eta|}}$ for each $\eta$.   

\para{1.7.}
For any $i \in I$ the braid group action $T_i : \BU_q \to \BU_q$ is defined 
as in [L4, \S 39] (denoted by $T''_{s_i,1}$ there). 
Let $\Bh = (i_1, \dots, i_{\nu})$ be a sequence such that 
$w_0 = s_{i_1}\cdots s_{i_{\nu}}$
is a reduced expression. 
For $\Bc = (c_1, \dots, c_{\nu}) \in \BN^{\nu}$, put
\begin{equation*}
\tag{1.7.1}
L(\Bc, \Bh) = f_{i_1}^{(c_1)}T_{i_1}(f_{i_2}^{(c_2)})\cdots 
                      (T_{i_1}\cdots T_{i_{{\nu}-1}})(f_{\nu}^{(c_{\nu})}).
\end{equation*}
Then $\{ L(\Bc, \Bh) \mid \Bc \in \BN^{\nu} \}$ gives a PBW-basis of 
$\BU_q^-$, which we denote by $\SX_{\Bh}$. 
Now assume given $\s : I \to I$ as in 1.2.
Then $\s\circ T_i \circ \s\iv = T_{\s(i)}$ and $T_iT_j = T_jT_i$ if $i,j \in \eta$. 
Hence one can define $R_{\eta} = \prod_{i \in \eta}T_i$ for each $\eta \in \ul I$, 
and $R_{\eta}$ commutes with $\s$.  
\par
We consider the braid group action $\ul T_{\eta} : \ul\BU_q \to \ul\BU_q$. 
Let $\ul\Bh = (\eta_1, \dots, \eta_{\ul \nu})$ be a sequence for $\ul w_0$. 
For any $\ul\Bc = (\g_1, \dots \g_{\ul \nu}) \in \BN^{\ul \nu}$, 
$L(\ul\Bc, \ul\Bh)$ is defined in a similar way as in  
(1.7.1), 
\begin{equation*}
\tag{1.7.2}
L(\ul\Bc, \ul\Bh) = \ul f_{\eta_1}^{(\g_1)}\ul T_{\eta_1}(\ul f_{\eta_2}^{(\g_2)})\cdots 
         (\ul T_{\eta_1}\cdots \ul T_{\eta_{\ul \nu -1}})(\ul f_{\eta_{\ul \nu}}^{(\g_{\ul \nu})}).
\end{equation*}
Then $\{ L(\ul\Bc, \ul\Bh) \mid \ul\Bc \in \BN^{\ul \nu}\}$ gives a PBW-basis of 
$\ul\BU_q^-$, which we denote by $\ul\SX_{\ul\Bh}$. 
\par
Now assume that $\Bh$ is obtained from $\ul\Bh$ as in 1.6. 
Then $L(\Bc, \Bh)$ can be written as follows.
For $k = 1, \dots, \ul \nu$, let $I_k$ be the interval in $[1, \nu]$ 
corresponding to $\eta_k$ so that $w_{\eta_k} = \prod_{j \in I_k}s_{i_j}$
in the expression of $w_0$ in (1.6.1).  
Put $F_{\eta_k}(\Bc) = \prod_{j \in I_k}f_{i_j}^{(c_j)}$ for each $k$. 
Then we have
\begin{equation*}
\tag{1.7.3}
L(\Bc, \Bh) = F_{\eta_1}(\Bc)R_{\eta_1}(F_{\eta_2}(\Bc))\cdots
                  (R_{\eta_1}\cdots R_{\eta_{\ul \nu -1}})(F_{\eta_{\ul \nu}}(\Bc)).
\end{equation*}
In particular, the following holds.
\begin{lem}   
Under the notation as above, 
\begin{enumerate}
\item 
\ $\s$ gives a permutation of the PBW-basis $\SX_{\Bh}$, 
namely $\s(L(\Bc, \Bh)) = L(\Bc', \Bh)$ for some $\Bc' \in \BN^{\nu}$. 
$L(\Bc, \Bh)$ is $\s$-invariant if and only if 
$c_j$ is constant for $j \in I_k$ for $k = 1, \dots, \ul \nu$.  
\item
For each $\ul\Bc \in \BN^{\ul\nu}$, let $\Bc \in \BN^{\nu}$ 
be the unique element such that $c_j = \g_k$ for each $j \in I_k$.
Then $L(\ul\Bc, \ul\Bh) \mapsto L(\Bc, \Bh)$ gives a bijection 
\begin{equation*}
\ul\SX_{\ul\Bh} \isom \SX_{\Bh}^{\s},
\end{equation*} 
where $\SX^{\s}_{\Bh}$ is the set of $\s$-stable PBW-basis in $\SX_{\Bh}$. 
\end{enumerate}
\end{lem}

\para{1.9.}
For each $\eta \in \ul I$ and $a \in \BN$, put 
$\wt f_{\eta}^{(a)} = \prod_{i \in \eta}f_i^{(a)}$. 
Since $f_i^{(a)}$ and $f_j^{(a)}$ commute each other for $i,j \in \eta$,  
we have $\wt f^{(a)}_{\eta} \in {}_{\BA}\BU_q^{-,\s}$. We denote its image 
in ${}_{\BA'}\BU_q^{-,\s}$ also by $\wt f^{(a)}_i$.  
Thus we can define $g_{\eta}^{(a)} \in \BV_q$ by 
\begin{equation*}
\tag{1.9.1}
g^{(a)}_{\eta} = \pi(\wt f_{\eta}^{(a)}). 
\end{equation*}
In the case where $a = 1$, we put 
$\wt f_{\eta}^{(1)} = \wt f_{\eta} = \prod_{i \in \eta}f_i$ 
and $g^{(1)}_{\eta} = g_{\eta}$. 
Recall that ${}_{\BA'}\ul\BU_q^-$ is generated by 
$\ul f^{(a)}_{\eta}$ for $\eta \in \ul I$ and $a \in \BN$. 
We have the following result.

\begin{prop}  
The correspondence $\ul f_{\eta}^{(a)} \mapsto g_{\eta}^{(a)}$ gives rise to 
a homomorphism  $\Phi : {}_{\BA'}\ul\BU^-_q \to \BV_q$ of $\BA'$-algebras. 
\end{prop}

\para{1.11.}
Proposition 1.10 will be proved in Section 3. Here assuming the proposition, 
we continue the discussion. Let $\SX_{\Bh}$ be as in Lemma 1.8.
It is known that the PBW -basis $\SX_{\Bh}$ is contained in  
${}_{\BA}U_q^-$ (see Introduction). 
Thus $\s$-stable PBW-basis $L(\Bc,\Bh)$ in $\SX^{\s}_{\Bh}$ 
is contained in ${}_{\BA}\BU_q^{-,\s}$. 
By Lemma 1.8 such an $L(\Bc, \Bh)$ can be written as
\begin{equation*}
\tag{1.11.1}
L(\Bc, \Bh) = \wt f_{\eta_1}^{(\g_1)}R_{\eta_1}(\wt f_{\eta_2}^{(\g_2)})
         \cdots (R_{\eta_1}\cdots R_{\eta_{\ul \nu - 1}})
                        (\wt f_{\eta_{\ul \nu}}^{(\g_{\ul \nu})}),
\end{equation*}
where $\ul\Bc = (\g_1, \dots, \g_{\ul \nu})$ and 
\begin{equation*}
\tag{1.11.2}
\Bc = (c_1, \dots, c_{\nu}) = (\underbrace{\g_1, \dots, \g_1}_{|\eta_1|\text{-times}}, 
                  \underbrace{\g_2, \dots, \g_2}_{|\eta_2|\text{-times}}, \cdots, 
       \underbrace{\g_{\ul \nu}, \dots, \g_{\ul \nu}}_{|\eta_{\ul \nu}|\text{-times}}). 
\end{equation*}
For each $L(\Bc, \Bh) \in \SX^{\s}_{\Bh}$, put 
$E(\ul\Bc, \ul\Bh) = \pi(L(\Bc, \Bh))$ under the correspondence in 
(1.11.2). 
By Lemma 1.8 (i), any element $x \in {}_{\BA'}\BU_q^{-,\s}$ can be written as 
an $\BA'$-linear combination of $\s$-stable PBW-basis modulo $J$. 
Thus we have 
\par\medskip\noindent
(1.11.3) \ 
The set $\{ E(\ul\Bc, \ul\Bh) \mid \ul\Bc \in \BN^{\ul \nu} \}$ generates $\BV_q$ 
as $\BA'$-module. 

\para{1.12.}
It is known, for any Cartan datum $X$, 
 that there exists a canonical symmetric bilinear form 
$(\ ,\ )$ on $\BU_q^-$, which satisfies the property, 
\begin{equation*}
\tag{1.12.1}
(L(\Bc, \Bh), L(\Bc',\Bh)) = \prod_{k = 1}^{\nu}(f_{i_k}^{(c_k)}, f_{i_k}^{(c'_k)})
                           = \prod_{k = 1}^{\nu}\d_{c_k,c'_k}
                              \prod_{d = 1}^{c_k}\frac{1}{1 - q_{i_k}}
\end{equation*}
for $\Bc = (c_1, \dots, c_{\nu}), \Bc' = (c_1', \dots, c_{\nu}')$. 
In particular, $(L(\Bc, \Bh), L(\Bc', \Bh)) = 0$ if $\Bc \ne \Bc'$, and 
the form $(\ ,\ )$ is non-degenerate. 
Assume that $X$ is as in 1.2. Then $\s$ preserves the form, namely, 
$(\s(x), \s(y)) = (x,y)$ for any $x,y \in \BU_q^-$. 
\par
Let $\BF(q)$ be the field of rational functions over $\BF$, 
and put ${}_{\BF(q)}\BV_q = \BV_q\otimes_{\BA'}\BF(q)$.
Then the form $(\ ,\ )$ on $\BU_q^-$ induces a symmetric bilinear form 
on ${}_{\BF(q)}\BV_q$ (note that $(\sum_i \s^i(x), \sum_i \s^i(y)) = 0$ in $\BF(q)$).   
We have $(E(\ul\Bc, \ul\Bh), E(\ul\Bc',\ul\Bh')) = 0$ if $\ul\Bc \ne \ul\Bc'$, 
and $(E(\ul\Bc, \ul\Bh), E(\ul\Bc, \ul\Bh)) \ne 0$. 
Thus $\{ E(\ul\Bc, \ul\Bh) \mid \ul\Bc \in \BN^{\ul \nu}\}$ gives rise 
to an orthogonal basis of ${}_{\BF(q)}\BV_q$. 
\par
Put ${}_{\BF(q)}\ul\BU_q^- = {}_{\BA'}\ul\BU_q^-\otimes_{\BA'}\BF(q)$. 
We can regard $\{ L(\ul\Bc, \ul\Bh) \mid \ul\Bc \in \BN^{\ul \nu}\}$ 
as an $\BF(q)$-basis of ${}_{\BF(q)}\ul\BU_q^{-}$. 
The map $\Phi : {}_{\BA'}\ul\BU_q^- \to \BV_q$ induces an algebra homomorphism 
${}_{\BF(q)}\ul\BU_q^- \to {}_{\BF(q)}\BV_q$, which we denote also by $\Phi$.
We need a lemma.  

\begin{lem}  
Assume that $\ul X$ has rank 2, and $\ul\Bh = (\eta_1, \dots, \eta_{\ul \nu})$.
Then for $k = 1, \dots, \ul \nu$, we have
\begin{equation*}
\tag{1.13.1}
\Phi(\ul T_{\eta_1} \dots \ul T_{\eta_{k-1}}(\ul f_{\eta_k})) 
   = \pi(R_{\eta_1}\cdots R_{\eta_{k-1}}(\wt f_{\eta_k})).
\end{equation*}
\end{lem}

\par\medskip
Lemma 1.13 will be proved in Section 4.  We continue the discussion 
assuming the lemma.  
By using Lemma 1.13, we can prove the following theorem.

\begin{thm}  
Let $\Bh$ and $\ul\Bh$ be as in 1.6.
\begin{enumerate}
\item 
For any $\ul\Bc \in \BN^{\ul \nu}$, 
we have $\Phi(L(\ul\Bc, \ul\Bh)) = E(\ul\Bc, \ul\Bh)$. 
\item 
$\Phi$ gives an algebra isomorphism ${}_{\BF(q)}\ul\BU_q^- \isom {}_{\BF(q)}\BV_q$.
\end{enumerate}
\end{thm}

\begin{proof}
Since $R_{\eta}$'s satisfy the braid relation, we can define 
$R_w = R_{\eta_1}\cdots R_{\eta_k}$ for a reduced expression  
$w = \ul s_{\eta_1}\cdots \ul s_{\eta_k} \in \ul W$.  
Let $\ul \vD^+$ be the set of positive roots in 
$\bigoplus_{\eta \in \ul I}\BQ\a_{\eta}$.
We consider the following statement 
\par\medskip\noindent
(1.14.1) \ Assume that $w(\a_{\eta}) \in \vD^+$.  Then 
$\pi\bigl(R_w(\wt f_{\eta})\bigr) = \Phi(\ul T_w(\ul f_{\eta}))$.  
\par\medskip
Note that (1.14.1) certainly holds in the case where $\ul X$ has rank 2, 
in view of Lemma 1.13.  
We prove (1.14.1) by induction on $l(w)$. (1.14.1) holds 
if $l(w) = 0$.  Thus we assume that $l(w) > 0$, and choose $\eta' \in \ul I$ 
such that $l(ws_{\eta'}) = l(w)-1$. From  the assumption in (1.14.1), 
$\eta' \ne \eta$. It is known that there exists $w',w'' \in \ul W$ such that 
$w = w'w''$, which satisfy the condition 
\par\medskip
(i) \ $w''$ is contained in the subgroup of $\ul W$ generated by $s_{\eta}$ 
    and $s_{\eta'}$, 
\par
(ii) \ $l(w) = l(w') + l(w'')$, 
\par
(iii) \ $l(w's_{\eta}) = l(w') +1$, $l(w's_{\eta'}) = l(w') + 1$.
\par\medskip
By applying (1.14.1) to the case $\ul X$ has rank 2, we see that 
$\pi(R_{w''}(\wt f_{\eta})) = \Phi(\ul T_{w''}(\ul f_{\eta}))$. 
Since $w \ne w'$, we have $l(w') < l(w)$.  Also note that 
$w'(\a_{\eta}), w'(\a_{\eta'}) \in \ul\vD^+$.  Thus by induction, we have
\begin{equation*}
\pi(R_{w'}(\wt f_{\eta})) = \Phi(\ul T_{w'}(\ul f_{\eta})), \quad
\pi(R_{w'}(\wt f_{\eta'})) = \Phi(\ul T_{w'}(\ul f_{\eta'})). 
\end{equation*}
Since $R_{w}(\wt f_{\eta}) = R_{w'}R_{w''}(\wt f_{\eta})$ and 
$\ul T_w(\ul f_{\eta}) = \ul T_{w'}\ul T_{w''}(\ul f_{\eta})$, 
(1.14.1) holds for $w$.  Thus (1.14.1) is proved. 
\par
Now the claim (i) in the theorem follows from (1.14.1). 
Let $Z$ be the $\BF(q)$-subspace of ${}_{\BF(q)}\ul\BU_q^-$ 
spanned by $\{ L(\ul\Bc, \ul\Bh)\}$.  
Since $\{ E(\ul\Bc, \ul\Bh)\}$ is a basis of ${}_{\BF(q)}\BV_q$,  
$\Phi$ gives an isomorphism  $Z \isom {}_{\BF(q)}\BV_q$ by (i), 
and so $Z$ is an algebra over $\BF(q)$.   
Since $\ul f_{\eta}^{(a)} = ([a]^!_{q_{\eta}})\iv \ul f_{\eta}^a$ is 
contained in $Z$, we see that $Z = {}_{\BF(q)}\ul\BU_q^-$.  
Thus (ii) holds.  The theorem is proved.
\end{proof}

\para{1.15.}
We follow the point of view explained in Introduction.
In the simply-laced case, the properties (i), (ii) and (iii) in (0.1.1) 
are known to hold. 
Hence there exists the canonical basis 
$\{ b(\Bc, \Bh) \mid \Bc \in \BN^{\nu} \}$ in $\SL_{\BZ}(\infty)$, 
which is characterized by the following properties, 
\begin{align*}
\tag{1.15.1}
\ol {b(\Bc, \Bh)} &= b(\Bc, \Bh), \\
b(\Bc, \Bh) &\equiv L(\Bc, \Bh) \mod q\SL_{\BZ}(\infty),  
\end{align*}
where $x \mapsto \ol{x}$ is the bar involution in $\BU_q^-$. 
Note that $\{ b(\Bc, \Bh) \mid \Bc \in \BN^{\nu}\}$ is independent of the choice of 
$\Bh$, which we denote by $\BB$.
\par
We define a total order on $\BN^{\nu}$ by making use of the lexicographic 
order, i.e., for $\Bc = (c_1, \dots, c_{\nu}), \Bd = (d_1, \dots, d_{\nu}) \in \BN^{\nu}$,
$\Bc < \Bd$ if and only if there exists $k$ such that $c_i = d_i$ for $i < k$ and 
$c_k < d_k$. 
Then the second formula in (1.15.1) can be written more precisely as 
\begin{equation*}
\tag{1.15.2}
b(\Bc,\Bh) = L(\Bc,\Bh) + \sum_{\Bc < \Bd}a_{\Bd}L(\Bd,\Bh)
\end{equation*}
with $a_{\Bd} \in q\BZ[q]$. 

\para{1.16.}
We choose $\Bh$ and $\ul\Bh$ as in 1.6. Since 
$\s$ permutes the PBW-basis $L(\Bc,\Bh)$, $\s$ permutes the canonical basis $\BB$.
We denote by $\BB^{\s}$ the set of $\s$-stable canonical basis of $\BU_q^-$. 
Take $\Bb = \Bb(\Bc, \Bh) \in \BB^{\s}$.  Then $L(\Bc, \Bh)$ is $\s$-stable, and 
$\Bc$ is obtained from $\ul\Bc$ as in 1.11. Since $\Bb \in {}_{\BA}\BU_q^{-,\s}$, 
one can consider $\pi(\Bb)$. 
Then we can write as 
\begin{equation*}
\tag{1.16.1}
\pi(\Bb) = E(\ul\Bc, \ul\Bh) + \sum_{\ul\Bc < \ul\Bd}a_{\ul\Bd}E(\ul\Bd, \ul\Bh)
\end{equation*}  
with $a_{\ul\Bd} \in q\BF[q]$.  
The total order $\ul\Bc < \ul\Bd$ on $\BN^{\ul\nu}$ is defined similarly.  
The bar involution can be defined on $\BV_q$, and the map $\pi$ is compatible 
with those bar involutions. Thus we have 
\begin{equation*}
\tag{1.16.2}
\ol{\pi(\Bb)} = \pi(\Bb).  
\end{equation*}
Let $\wt\SL_{\BF}(\infty)$ be the $\BF[q]$-submodule of $\BV_q$ generated by 
$E(\ul\Bc, \ul\Bh)$.  Then the set $\{ \pi(\Bb) \mid \Bb \in \BB^{\s}\}$ 
gives rise to an $\BF[q]$-basis of $\wt\SL_{\BF}(\infty)$ satisfying the properties 
(1.16.1) and (1.16.2). Note that the set $\{ \pi(\Bb) \mid \Bb \in \BB^{\s}\}$ 
is characterized by those properties, and this set is independent of the 
choice of $\ul\Bh$, which we call the canonical basis of $\BV_q$.
\par   
Let $\ul\SL_{\BF}(\infty)$ be the $\BF[q]$-submodule of ${}_{\BF(q)}\ul\BU_q^-$ 
generated by $\{ L(\ul\Bc, \ul\Bh) \mid \ul\Bc \in \BN^{\ul \nu} \}$.   
We have the following result.
\begin{prop}  
There exists a unique $\BF[q]$-basis 
$\{ \Bb(\ul\Bc, \ul\Bh) \mid \ul\Bc \in \BN^{\ul \nu}\}$ in $\ul\SL_{\BF}(\infty)$
satisfying the following properties, 
\begin{align*}
\tag{1.17.1}
\ol{\Bb(\ul\Bc, \ul\Bh)} &= \Bb(\ul\Bc, \ul\Bh), \\
\Bb(\ul\Bc, \ul\Bh) &= L(\ul\Bc, \ul\Bh) + \sum_{\ul\Bc < \ul\Bd}a_{\ul\Bd}L(\ul\Bd, \ul\Bh), 
\qquad (a_{\ul\Bd} \in q\BF[q]).
\end{align*}
Moreover, the set $\{ \Bb(\ul\Bc, \ul\Bh)\}$ is independent of the choice of $\ul\Bh$,  
and $\ul\SL_{\BF}(\infty)$ does not depend on the choice of $\ul\Bh$. 
\end{prop}

\begin{proof}
It is clear that the map $\Phi : {}_{\BF(q)}\ul\BU_q^- \to {}_{\BF(q)}\BV_q$ 
is compatible with the bar involutions.  
Then the proposition immediately follows from  Theorem 1.14. 
\end{proof}

\para{1.18.}
For any $\ul X$, we consider the following statements corresponding to 
(iii) and (iii$'$) in (0.1.1).
\par\medskip\noindent
(1.18.1) \ PBW-basis $\ul\SX_{\ul\Bh}$ gives an $\BA$-basis of 
${}_{\BA}\ul\BU_q^-$.
\par\medskip\noindent
(1.18.2) \ Any element $L(\ul\Bc, \ul\Bh) \in \ul\SX_{\ul\Bh}$ is 
contained in ${}_{\BA}\ul\BU_q^-$. 
\par\medskip

As was explained in Introduction, the proof of (1.18.1) is reduced to the
case of rank 2, namely the case of type $B_2$ and $G_2$, and in that case, 
(1.18.2) was proved by Lusztig [L1] and Xi [X1]. 
In any case, the computation in the case of $G_2$ is not easy.  
(1.18.2) can be proved by computing the commutation 
relations of root vectors, which is relatively easy compared to (1.18.1). 
\par
In the discussion below, we only assume that (1.18.2) 
holds for ${}_{\BA}\ul\BU_q^-$, 
and will prove that (1.18.1) holds for ${}_{\BA}\ul\BU_q^-$. 

\para{1.19.}
We return to our original setting, and consider the map 
$\Phi : {}_{\BA'}\ul\BU_q^- \to \BV_q$. 
By (1.18.2), the PBW-basis $\ul\SX_{\ul\Bh} = \{ L(\ul\Bc, \ul\Bh)\}$ 
is contained in ${}_{\BA'}\ul\BU_q^-$. 
Since $\{ E(\ul\Bc, \ul\Bh)\}$ is an $\BA'$-basis of $\BV_q$, we see that 
$\Phi$ is surjective, by Theorem 1.14 (i). 
Let ${}_{\BA'}\wt{\ul\BU}^-_q$ be the $\BA'$-module generated by 
$\{ L(\ul\Bc, \ul\Bh) \mid \ul\Bc \in \BN^{\ul \nu}\}$. 
Again by Theorem 1.14, ${}_{\BA'}\wt{\ul\BU}_q^-$ is an $\BA'$-submodule of 
${}_{\BF(q)}\ul\BU_q^-$, which is independent of the choice of $\ul\Bh$.     
We show that 
\begin{equation*}
\tag{1.19.1}
{}_{\BA'}\wt{\ul\BU}_q^- = {}_{\BA'}\ul\BU_q^-.
\end{equation*} 
By (1.18.2), we know that ${}_{\BA'}\wt{\ul\BU}_q^- \subset {}_{\BA'}\ul\BU_q^-$.
On the other hand, for each $\eta \in \ul I$, one can find a sequence 
$\ul\Bh = (\eta_1, \dots, \eta_{\ul N})$ such that $\eta_1 = \eta$. This implies 
that ${}_{\BA'}\wt{\ul\BU}_q^-$ is invariant under the left multiplication 
by $\ul f^{(a)}_{\eta}$. Since this is true for any $\eta$, we see that 
${}_{\BA'}\ul\BU_q^-$ is contained in ${}_{\BA'}\wt{\ul\BU}_q^-$.  Thus (1.19.1) holds. 
\par
Summing up the above arguments, we have the following integral form of 
Theorem 1.14.

\begin{prop} 
Assume that $(1.18.2)$ holds for ${}_{\BA}\ul\BU_q^-$.  Then $\Phi$ induces 
an isomorphism ${}_{\BA'}\ul\BU_q^- \simeq \BV_q$. 
In particular, the PBW-basis $\ul\SX_{\ul\Bh}$ gives an $\BA'$-basis of 
${}_{\BA'}\ul\BU_q^-$. 
\end{prop}

As a corollary, we have

\begin{cor}  
Assume that $(1.18.2)$ holds for ${}_{\BA}\ul\BU_q^-$.  Then 
$(1.18.1)$ also holds. 
\end{cor} 

\begin{proof}
Let ${}_{\BA}\wh{\ul\BU}_q^-$ be the inverse limit of 
${}_{\BA}\ul\BU_q^-/\ve^n({}_{\BA}\ul\BU_q^-)$.  Then ${}_{\BA}\wh{\ul\BU}_q^-$ 
has a natural structure of the module over 
$\displaystyle\BZ_{\ve}[q,q\iv] = \lim_{\leftarrow}\BA/\ve^n\BA$, 
where $\BZ_{\ve}$ is the ring of $\ve$-adic integers. 
We have a natural embedding ${}_{\BA}\ul\BU_q^- \subset {}_{\BA}\wh{\ul\BU}_q^-$. 
Now take $x \in {}_{\BA}\ul\BU_q^-$. (1.18.1) shows that $x$ can be written as 
a linear combination of PBW-basis with coefficients in $\BA$ modulo $\ve({}_{\BA}\ul\BU_q^-)$. 
We regard $x$ as an element in ${}_{\BA}\wh{\ul\BU}_q^-$.  Then 
$x$ can be written as a linear combination of PBW-basis with coefficients 
in $\BZ_{\ve}[q,q\iv]$.   
On the other hand, we know that $x$ is a linear combination of PBW-basis with 
coefficients in $\BQ(q)$. 
Thus those coefficients belong to $\BA = \BZ[q,q\iv]$, and we obtain (1.18.1). 
\end{proof}

\para{1.22.}
We assume that (1.18.2) holds for $\ul\BU_q^-$.  Then by Corollary 1.21, 
we have
\par\medskip\noindent
(1.22.1) \ In $\ul\BU_q^-$, $\ol{L(\ul\Bc, \ul\Bh)}$ is a linear combination of
various $L(\ul\Bd, \ul\Bh)$ with coefficients in $\BA$. 
\par\medskip
Then by [L3, Lemma 24.2.1], one can define a basis 
$\{ \Bb(\ul\Bc, \ul\Bh) \mid \ul\Bc \in \BN^{\ul\nu}\}$
of $\ul\BU_q^-$, satisfying the properties 
\begin{align*}
\tag{1.22.2}
\ol{\Bb(\ul\Bc, \ul\Bh)} &= \Bb(\ul\Bc, \ul\Bh), \\
\Bb(\ul\Bc,\ul\Bh) &= L(\ul\Bc, \ul\Bh) + \sum_{\ul\Bc < \ul\Bd}a_{\ul\Bd}L(\ul\Bd, \ul\Bh), 
\qquad (a_{\ul\Bd} \in q\BZ[q]).
\end{align*}
In this construction, we cannot give the independence of the basis 
$\{ \Bb(\ul\Bc, \ul\Bh)\}$ from $\ul\Bh$.  But by using the almost orthogonality 
of PBW-basis (1.12.1), one can prove a weaker property, namely, 
the independence from $\ul\Bh$, up to sign (see [L3, Thm. 14.2.3]); 
if we fix $\Bh, \Bh'$, then for any $\ul\Bc$, there exists a unique $\ul\Bc'$
such that   
\begin{equation*}
\tag{1.22.3}
\Bb(\ul\Bc, \ul\Bh) = \pm \Bb(\ul \Bc', \ul\Bh').
\end{equation*}   

We denote by $\ul\BB$ the set of canonical basis $\{ b(\ul\Bc, \Bh) \}$ 
in $\ul\BU_q^-$.
On the other hand, let $\ul\BB'$ be the canonical basis in 
${}_{\BA'}\ul\BU_q^-$ given 
in Proposition 1.17. We temporally write them as $\{ \Bb'(\ul\Bc, \ul\Bh) \}$. 
Then the image of $\Bb(\ul\Bc, \ul\Bh)$ under the
natural map ${}_{\BA}\ul\BU_q^- \to {}_{\BA'}\ul\BU_q^-$ coincides with 
$\Bb'(\ul\Bc, \ul\Bh)$, and this gives a bijection $\ul\BB \isom \ul\BB'$.  
In the case where $\ve = 2$, this does not give a new information on the sign 
of $\Bb(\ul\Bc, \ul\Bh)$.
But in the case where $\ve = 3$, we have the following result.

\begin{prop}  
Assume that $\ve = 3$, and  $\ul X$ is of type $G_2$. 
Then the canonical basis $\{ \Bb(\ul\Bc, \ul\Bh) \mid \ul\Bc \in \BN^{\ul\nu}\}$ 
is independent of the choice 
of $\ul\Bh$, namely, if we fix $\ul\Bh, \ul\Bh'$, then for any $\ul\Bc$, there exists 
a unique $\ul\Bc'$ such that 
\begin{equation*}
\Bb(\ul\Bc, \ul\Bh) = \Bb(\ul\Bc', \ul\Bh').
\end{equation*}  
\end{prop}

\begin{proof}
By (1.22.3), we have $\Bb'(\ul\Bc, \ul\Bh) = a \Bb'(\ul\Bc',\ul\Bh')$ 
for some $a = \pm 1$. But by Proposition 1.17, $\Bb'(\ul\Bc, \ul\Bh)$ 
is determined uniquely as an element in $\ul\SL_{\BF}(\infty)$, which is 
independent of the choice of $\ul\Bh$.  It follows that $a \equiv 1 \mod 3$. 
This implies that $a = 1$, and the proposition is proved. 
\end{proof}

\remark{1.24.}
By Proposition 1.20, we have a natural bijection $\ul\BB' \isom \BB^{\s}$. 
By the discussion in 1.22, we have $\ul\BB \simeq \ul\BB'$. Hence 
\begin{equation*}
\tag{1.24.1}
\BB^{\s} \simeq \ul\BB' \simeq \ul\BB.
\end{equation*}

Thus we have a natural correspondence $\BB^{\s} \lra \ul \BB$ 
between the set of $\s$-stable canonical basis of $\BU_q^-$ and   
the set of canonical basis of $\ul\BU_q^-$.  
This is nothing but the reformulation, by our context of elementary setting, 
of Lusztig's result [L4, 1.12 (b)](see also [L3, Thm. 14.4.9]) obtained by 
geometric considerations. 

\par\bigskip
\section{PBW-bases for affine quantum groups}

\para{2.1.}
In Beck and Nakajima [BN], the PBW-bases were constructed 
in the case of affine quantum groups.  In this section, by making use of 
their PBW-bases, we shall 
extend the results in the previous section to the case of affine quantum 
groups. 
\par
Let $\Fg$ be an untwisted affine Lie algebra associated to the simply laced 
Cartan datum $X$, with the vertex set $I$, 
and $\Fg_0$ the simple Lie algebra over $\BC$ with the vertex set $I_0$ associated to 
the simply laced Cartan datum $X_0$ such that 
\begin{align*}
L\Fg_0 &= \Fg_0\otimes_{\BC}\BC[t, t\iv], \\
\Fg &= L\Fg_0 \oplus \BC c \oplus \BC d,
\end{align*} 
where $c$ is the center of $\Fg$ and $d$ is the degree operator.
Here $L\Fg_0 \oplus \BC c$ is the central extension of 
the Loop algebra $L\Fg_0$.
\par
Let $\Fg_0 = \Fh_0 \oplus \bigoplus_{\a \in \vD_0}(\Fg_0)_{\a}$ be the root space 
decomposition of $\Fg_0$ with respect to a Cartan subalgebra $\Fh_0$ of $\Fg_0$, 
where $\vD_0$ is the set of roots in $\Fg_0$. 
Then $\Fh = \Fh_0 \oplus \BC c \oplus \BC d$ is a Cartan subalgebra of $\Fg$, and
$\Fg$ is decomposed as 
\begin{equation*}
\tag{2.1.1}
\Fg = \Fh \oplus 
    \biggl(\bigoplus_{\substack{\a \in \vD_0 \\ m \in \BZ}}(\Fg_0)_{\a} \otimes t^m\biggr)
         \oplus \biggl(\bigoplus_{m \in \BZ - \{0\}}\Fh_0 \otimes t^m\biggr).
\end{equation*} 
We define $\d \in \Fh^*$ by $\lp d, \d \rp = 1, \lp \Fh_0 \oplus \BC c, \d \rp = 0$. 
We regard $\a \in \vD_0 \subset \Fh_0^*$ as an element in $\Fh^*$ by 
$\a(c) = 0, \a(d) = 0$.
Then $(\Fg_0)_{\a} \otimes t^m, \Fh_0 \otimes t^m$ corresponds to the root space with 
root $\a + m\d$, $m\d$, respectively, and (2.1.1) gives a root space decomposition of 
$\Fg$ with respect to $\Fh$.  Let $\vD$ (resp. $\vD^+$) be the set of roots 
(resp. the set of positive roots) in $\Fg$.  
Also $\vD_0^+$ be the set of positive roots in $\vD_0$. 
Then $\vD^+$ is given by 
\begin{equation*}
\tag{2.1.2}
\vD^+ = \vD^{\re,+}_{>}  \sqcup \vD^{\re,+}_{<} \sqcup  \BZ_{>0}\d,    
\end{equation*}
where
\begin{align*}
\vD^{\re,+}_{>} &= \{ \a + m\d \mid \a \in \vD_0^+, m \in \BZ_{\ge 0}\},  \\ 
\vD^{\re,+}_{<} &= \{ \a + m\d \mid \a \in -\vD_0^+, m \in \BZ_{> 0} \}. 
\end{align*} 
$\vD^{\re,+}_{>} \sqcup \vD^{\re,+}_{<}$ is the set of positive real roots, 
and $\BZ_{>0}\d$ is the set of positive imaginary roots.
The simple roots $\Pi$ are given by 
\begin{equation*}
\Pi = \{ \a_i \mid i \in I_0 \} \sqcup \{ \a_0 = \d - \th \}, 
\end{equation*}
where $\th$ is the highest root in $\vD_0^+$.

\para{2.2.}
Let $\s : I \to I$ be the permutation as in 1.2. 
We assume that $\s$ preserves $I_0$.  
Thus if $X$ is irreducible, 
$X$ has type $A_{2n+1}^{(1)} (n \ge 1), D_n^{(1)} (n \ge 4), E_6^{(1)}$ 
for $\ve = 2$, and 
$D_4^{(1)}$ for $\ve = 3$. Correspondingly, $\ul X$ has type 
$D^{(2)}_{n+2}, (n \ge 1),  A^{(2)}_{2n -3}, (n \ge 4),  E_6^{(2)}$ and $D_4^{(3)}$
under the notation of the table in [Ka, Section 4.8].  
Let $\ul I_0$ be the set of $\s$-orbits in $I_0$, and $\ul X_0$ be the corresponding 
Cartan datum.  Then $\ul X_0$ has type $B_{n+1}, C_{n-1}, F_4, G_2$, respectively.  
\par
$\s$ induces a Lie algebra automorphism $\s : \Fg \to \Fg$, and
let $\Fg^{\s}$ be the subalgebra of $\Fg$ consisting of $\s$-fixed elements. 
$\s$ preserves $\Fg_0$, and $\s(c) = c, \s(d) = d$.  We define $\Fg_0^{\s}$ similarly.
Then $\Fg^{\s}_0$ is a simple Lie algebra, and 
$\Fg^{\s} = L\Fg^{\s}_0 \oplus \BC c \oplus \BC d$ is the affine Lie algebra 
associated to $\Fg^{\s}_0$. Note that $\Fg^{\s}$ is isomorphic to the 
affine Lie algebra $\ul \Fg$ associated to $\ul X$, 
which is the twisted affine Lie algebra of type $X^{(r)}_k$ given above
(here $r$ coincides with $\ve$). 
Moreover $\Fg_0^{\s}$ is isomorphic to 
$\ul \Fg_0$ associated to $\ul X_0$. 
We have $\Fh^{\s} = \Fh^{\s}_0 \oplus \BC c \oplus \BC d$, and $\Fh^{\s} \simeq \ul\Fh$, 
$\Fh_0^{\s} \simeq \ul\Fh_0$ (Cartan subalgebras of $\ul\Fg$ and $\ul\Fg_0$).     
\par
Note that $\s$ acts on $\vD^+$, leaving $\vD_0^+$ invariant.  
Moreover, $\s(\d) = \d$. Thus $\vD^{\re,+}_{>}$ and $\vD^{\re,+}_{<}$
are stable by $\s$. 
\par
Let $\ul\vD^+$ (resp. $\ul\vD^{\re,+}$, $\ul\vD^{\ima,+}$) be the set of positive roots (resp. 
positive real roots, positive imaginary roots) in the root system $\ul\vD$ of $\ul\Fg$. 
Since $\ul\Fg$ is twisted of type $X^{(r)}_n$, by [K, Prop.6.3], 
$\ul\vD^{\re,+}$ can be written as 
$\ul\vD^{\re,+} = \ul\vD^{\re,+}_{>} \sqcup \ul\vD^{\re,+}_{<}$ and 
$\ul\vD^{\ima,+} = \BZ_{>0}\d$, where 
\begin{align*}
\tag{2.2.1}
\ul\vD^{\re,+}_{>} &= \{\a + m\d \mid \a \in (\ul\vD^+_0)_s, m \in \BZ_{\ge 0} \} \sqcup
            \{ \a + mr\d \mid \a \in (\ul\vD^+_0)_l, m \in \BZ_{\ge 0} \},  \\ 
\ul\vD^{\re,+}_{<} &= \{\a + m\d \mid \a \in -(\ul\vD^+_0)_s, m \in \BZ_{> 0} \} \sqcup
            \{ \a + mr\d \mid \a \in -(\ul\vD^+_0)_l, m \in \BZ_{>0} \}. 
\end{align*}  
Here $(\ul\vD^+_0)_s$ (resp. $(\ul\vD^+_0)_l$) is the set of positive short roots 
(resp. positive long roots) in the root system $\ul\vD_0$ of $\ul\Fg_0$. 

\para{2.3.}
Let $\Fh^{*0} = \{ \la \in \Fh^* \mid \lp c, \la \rp = 0\}
              = \{ \la \in \Fh^* \mid (\la, \d) = 0\}$ be the subspace of 
$\Fh^*$.  Then $\Fh^{*0} = \bigoplus_{i \in I_0}\BC\a_i \oplus \BC \d$. 
We define a map 
\begin{equation*}
\cl : \Fh^{*0} \to \Fh_0^*
\end{equation*} 
by $\cl(\a_i) = \a_i \ (i \in I_0)$ and $\cl(\d) = 0$, where 
$\Fh_0^* = (\Fh_0)^*$.
Then $\cl$ induces an isomorphism $\Fh^{*0}/\BC\d \isom \Fh_0^*$.  
$\s$ acts on $\Fh^{*0}$ and on $\Fh_0^*$, and $\cl$ is compatible with those 
$\s$-actions. Hence $\cl$ induces a map $(\Fh^{*0})^{\s} \to (\Fh_0^*)^{\s}$. 
The restriction map $\Fh^* \to (\Fh^{\s})^*$ induces an isomorphism 
$(\Fh^*)^{\s} \isom (\Fh^{\s})^*$,  
which implies that $(\Fh^{*0})^{\s} \simeq \ul\Fh^{*0}$ since $\Fh^{\s} \simeq \ul\Fh$. 
Similarly we have $(\Fh_0^*)^{\s} \simeq (\Fh_0^{\s})^* \simeq \ul\Fh_0^*$. 
Under those identifications, the induced map $(\Fh^{*0})^{\s} \to (\Fh_0^*)^{\s}$ 
coincides with the map $\cl : \ul \Fh^{*0} \to \ul \Fh_0^*$ defined for $\ul\Fg$ 
similarly to $\Fg$.  

\para{2.4.}
Let $Q_{\cl}$ be the image of $\bigoplus_{i \in I_0}\BZ \a_i$ in $\Fh^{*0}/\BC\d$.
Then $Q_{\cl}$ can be identified with the root lattice
of $\Fg_0$ via $\cl$.  
We define $t : \Fh^{*0} \to GL(\Fh^*)$ by 
\begin{equation*}
\tag{2.4.1}
t(\xi)(\la) = \la + (\la, \d)\xi - \biggl\{ (\la, \xi) + 
                \frac{(\xi,\xi)}{2}(\la,\d)\biggr\}\d,  
\quad (\xi \in \Fh^{*0}, \la \in \Fh^*), 
\end{equation*}
which induces a map $t : \Fh^{*0}/\BC\d \to GL(\Fh^*)$, 
and consider the restriction of $t$ on $Q_{\cl}$. 
Note that in the case where $\la \in \Fh^{*0}$, (2.4.1) can be written 
in a simple form
\begin{equation*}
\tag{2.4.2}
t(\xi)(\la) = \la - (\la, \xi)\d.
\end{equation*}
Let $W$ be the Weyl group of $\Fg$ and $W_0$ the Weyl group of $\Fg_0$. 
Then we have an exact sequence
\begin{equation*}
\tag{2.4.3}
\begin{CD}
1 @>>> Q_{\cl} @>t>> W  @>>>  W_0 @>>> 1.
\end{CD}
\end{equation*}

Put 
\begin{equation*}
P_{\cl} = \{ \la \in \Fh^{*0} \mid (\la, \a_i) \in \BZ \text{ for any } 
                   i \in I_0 \}/\BC\d.
\end{equation*}
Then $P_{\cl}$ is identified with the weight lattice of $\Fg_0$ via $\cl$.
We define an extended affine Weyl group $\wt W$ by 
$\wt W = P_{\cl}\rtimes W_0$ (note that $\Fg$ is simply laced). 
\par
Let $\ul W$ be the Weyl group of $\ul\Fg$ and $\ul W_0$ the Weyl group of 
$\ul\Fg_0$.
Let $(\ ,\ )_1$ be the non-degenerate symmetric bilinear form  on  $\ul\Fh^*_0$, 
normalized that 
$(\a_i, \a_i)_1 = 2$ for a short root $\a_i$ ($i \in \ul I_0)$ (see 1.2).
The form $(\ ,\ )_1$ is extended uniquely to a non-degenerate symmetric bilinear form 
$(\ , \ )_1$ on $\ul\Fh^*$ by the condition that 
$(\la,\d) = \lp c, \la \rp$ for any $\la \in \ul\Fh^*$. 
For $\a \in \ul\vD_0$, put $\a^{\vee} = 2\a/(\a,\a)_1$.  
Put $\ul Q_{\cl} = \bigoplus_{\eta \in \ul I_0}\BZ \a_{\eta}$ and 
$\ul Q^{\vee}_{\cl} = \bigoplus_{\eta \in \ul I_0}\BZ\a_{\eta}^{\vee}$.   
Since $\ul\Fg$ is the dual of the untwisted algebra, 
we have $\ul Q_{\cl} \subset \ul Q^{\vee}_{\cl}$. 
As in (2.4.1), we can define a map 
$t : \ul\Fh^{*0}/\BC\d \to GL(\ul\Fh^*)$, and we have an exact sequence
\begin{equation*}
\tag{2.4.4}
\begin{CD}
1 @>>> \ul Q_{\cl}^{\vee} @>t>>  \ul W @>>> \ul W_0 @>>> 1.
\end{CD}
\end{equation*}

\par
For each $i \in I_0$, let $\w_i$ be the fundamental weight of 
$(\vD_0, \Fh_0^*)$, defined by $(\w_i, \a_j) = \d_{ij}$ ($i,j \in I_0$). 
Then under the isomorphism $\cl : \Fh^{*0}/\BC\d \simeq \Fh_0^*$, 
$P_{\cl} \simeq \bigoplus_{i \in I_0}\BZ \w_i$. 
The action of $\s$ on $\Fh^{*0}$ induces an action of $\s$ on $P_{\cl}$, which 
is given by $\w_i \mapsto \w_{\s(i)}$ $(i \in I_0)$. 
Thus we have an action of $\s$ on $\wt W$, which preserves $W_0$. 
On the other hand, we define the fundamental coweight $\ul \w^{\vee}_{\eta}$ of 
$(\ul\vD_0, \ul\Fh_0^*)$ by $(\ul \w^{\vee}_{\eta}, \a_{\eta'})_1 = \d_{\eta\eta'}$
($\eta, \eta' \in \ul I_0$), and put 
$\ul{\wt \w}_{\eta} = |\eta|\ul\w^{\vee}_{\eta}$.    
We define $\ul{\wt P}_{\cl} = \bigoplus_{\eta \in \ul I_0}\BZ\wt{\ul\w}_{\eta}$, 
which we regard as a lattice of $\ul\Fh^{*0}/\BC\d$ dual to $\ul Q^{\vee}_{\cl}$. 
Define the extended affine Weyl group by 
$\ul{\wt W} = \ul{\wt P}_{\cl} \rtimes \ul W_0$.  Since the map
\begin{equation*}
\tag{2.4.5}
(P_{\cl})^{\s} \isom \ul{\wt P}_{\cl}, \quad  \sum_{i \in \eta}\w_i 
          \mapsto |\eta|\ul\w^{\vee}_{\eta} = \wt{\ul\w}_{\eta}  
\end{equation*}
is compatible with the action of $W_0^{\s}\simeq \ul W_0$, 
we have an isomoprhism 
\begin{equation*}
\tag{2.4.6}
\wt W^{\s} = (P_{\cl})^{\s} \rtimes W_0^{\s} 
         \simeq \wt{\ul P}_{\cl} \rtimes \ul W_0 = \wt{\ul W}.
\end{equation*}

Let $\ST = \{ w \in \wt W \mid w(\vD^+) \subset \vD^+ \}$,which is a subgroup of 
the automorphism group of the ambient diagram.  Then we have 
$\wt W = \ST \ltimes W$. 
Similarly we define 
$\ul\ST = \{ w \in \ul W \mid w(\ul\vD^+) \subset \ul\vD^+\}$ so that. 
$\wt{\ul W} = \ul\ST \ltimes \ul W$. The action of $\s$ on $\wt W$ 
preserves $\ST$, and we have $\ST^{\s} = \ul \ST$. 

\para{2.5.}
Following [BN, 3.1], put
\begin{equation*}
\tag{2.5.1}
\xi = \sum_{i \in I_0}\w_i \in P_{\cl}, 
\end{equation*}
and consider $t(\xi) \in \wt W$, which we simply denote by $\xi$.
Here $\xi \in \wt W^{\s} \simeq \wt{\ul W } = \ul\ST \ltimes \ul W$, and 
one can express $\xi$ as
\begin{equation*}
\tag{2.5.2}
\xi = s_{\eta_1}\cdots s_{\eta_{\ul \nu}}\tau
\end{equation*} 
with $\tau \in \ul\ST = \ST^{\s}$, where $w = s_{\eta_1}\cdots s_{\eta_{\ul \nu}}$
is a reduced expression of $w \in \ul W$ ($w$ is the $\ul W$-component of $\xi$). 
Accordingly, we obtain a reduced 
expression of $w = s_{i_1} \cdots s_{i_{\nu}} \in W$ such that
\begin{equation*}
\tag{2.5.3}
w = \biggl(\prod_{k_1 \in \eta_1}s_{k_1}\biggr)\cdots 
    \biggl(\prod_{k_{\ul \nu} \in \eta_{\ul \nu}}s_{k_{\ul \nu}}\biggr) 
  = s_{i_1}\cdots s_{i_{\nu}}. 
\end{equation*}  

\par
As in [BN, (3.1)], we define a doubly infinite sequence attached to $\Fg$ 
\begin{equation*}
\tag{2.5.4}
\Bh = (\dots i_{-1}, i_0, i_1, \dots)
\end{equation*}
 by 
setting $i_{k + \nu} = \tau(i_k)$ for $k \in \BZ$. 
Then for any integer $m < p$, the product $s_{i_m}s_{i_{m +1}} \cdots s_{i_p} \in W$ 
is a reduced expression. Similarly, we define a doubly infinite sequence 
\begin{equation*}
\tag{2.5.5}
\ul \Bh = (\dots, \eta_{-1}, \eta_0, \eta_{1}, \dots)
\end{equation*}
by the condition that
$\eta_{k + \ul \nu} = \tau(\eta_k)$ for $k \in \BZ$, which satisfies the property 
that $s_{\eta_m}s_{\eta_{m+1}}\cdots s_{\eta_p} \in \ul W$ 
is a reduced expression for $m < p$.      
Note that $\xi \in (P_{\cl})^{\s}$, and under the isomorphism 
$(P_{\cl})^{\s} \simeq \wt{\ul P}_{\cl}$ in (2.4.5), $\xi$ coincides with  
the element $\sum_{\eta \in \ul I_0}\wt{\ul\w}_{\eta}$. 
Thus the sequence (2.5.5) is exactly the sequence defined in [BN, 3.1] 
attached to $\ul\Fg$.  

\par
By (2.4.2), for $\b = \a + m\d \in \vD^{\re,+}_{>}$ and $n \in \BZ$, 
$(n\xi)\iv (\b) = \b + n(\xi, \b)\d = \a + (m + n(\xi,\b))\d$.  Since 
$(\xi,\b) > 0$ by (2.5.1), $(n\xi)\iv(\b) \in \vD^-$ if $n < 0$ is small enough.   
Similar argument holds also for $\b \in \vD^{\re,+}_{<}$ by replacing 
$n <  0$ by $n >  0$.
It follows that 
\begin{align*}
\tag{2.5.6}
\bigcup_{n \in \BZ_{< 0}}(\vD^{\re,+}_{>} \cap w^n(\vD^-)) = \vD^{\re,+}_{>}, 
\qquad 
\bigcup_{n \in \BZ_{> 0}}(\vD^{\re,+}_{<} \cap w^n(\vD^-)) = \vD^{\re,+}_{<}. 
\end{align*}
Similar formulas hold also for the root system $\ul \vD^+$ of $\ul\Fg$. 
As a corollary of (2.5.6), we have
\par\medskip\noindent
(2.5.7) \  Let $\Bh$ be as in (2.5.4).  Then any $i \in I$ appears 
in the infinite sequence $(\dots i_{-1}, i_0, i_1, \dots)$.  Similarly, let 
$\ul\Bh$ be as in (2.5.5).  Then any $\eta \in \ul I$ appears 
in the infinite sequence $(\dots, \eta_{-1}, \eta_0, \eta_{1}, \dots)$.   

\para{2.6.}
Let $\BU_q^-$ (resp. $\ul\BU_q^-$) be the negative part of the quantum 
enveloping algebra $\BU_q$ (resp $\ul\BU_q$) 
 associated to $X$ (resp. $\ul X$).  We follow the notation 
in 1.5.
We fix $\Bh$ as in (2.5.4), and define $\b_k \in \vD^+$ for $k \in \BZ$ by 
\begin{equation*}
\tag{2.6.1}
\b_k = \begin{cases}
          s_{i_0}s_{i_{-1}}\cdots s_{i_{k + 1}}(\a_{i_k})  &\text{ if $k \le 0$}, \\ 
          s_{i_1}s_{i_2} \cdots s_{i_{k-1}}(\a_{i_k})     &\text{ if $k > 0$}. 
       \end{cases}
\end{equation*}
Then, as in [BN, 3.1], we have
\begin{align*}
\tag{2.6.2}
\vD^{\re,+}_{>} = \{ \b_k \mid k \in \BZ_{\le 0} \}, \qquad
\vD^{\re,+}_{<} = \{ \b_k \mid k \in \BZ_{> 0}  \}.
\end{align*}

We define root vectors $f^{(c)}_{\b_k} \in \BU_q^-$ by 

\begin{equation*}
\tag{2.6.3}
    f^{(c)}_{\b_k} = \begin{cases}
                  T_{i_0}T_{i_{-1}}\cdots T_{i_{k+1}}(f^{(c)}_{i_k}), 
                       &\quad\text{ if } k \le 0, \\
                  T_{i_1}\iv T_{i_2}\iv \cdots T_{i_{k-1}}\iv (f^{(c)}_{i_k}), 
                       &\quad\text{ if } k > 0. 
               \end{cases}  
\end{equation*}
We fix $p \in \BZ$, and let 
$\Bc_{+_p} = (c_p, c_{p-1}, \dots) \in \BN^{\BZ_{\le p}}, 
\Bc_{-_p} = (c_{p+1}, c_{p+2}, \dots) \in \BN^{\BZ_{> p}}$
be functions which are almost everywhere 0. 
We define $L(\Bc_{+_p}), L(\Bc_{-_p}) \in \BU_q^-$ by 
\begin{align*}
\tag{2.6.4}
L(\Bc_{+_p}) &= f_{i_p}^{(c_p)}T_{i_p}(f^{(c_{p-1})}_{i_{p-1}})
        T_{i_p}T_{i_{p-1}}(f^{(c_{p-2})}_{i_{p-2}})  \cdots \\ 
L(\Bc_{-_p}) &= \cdots T_{i_{p+1}}\iv T_{i_{p+2}}\iv(f_{i_{p+3}}^{(c_{p+3})}) 
         T_{i_{p+1}}\iv (f_{i_{p+2}}^{(c_{p+2})})f_{i_{p+1}}^{(c_{p+1})}.
\end{align*}
In the case where $p = 0$, 
we simply write $\Bc_{+_p}, \Bc_{-_p}$ as $\Bc_+, \Bc_-$. 
Thus $(\Bc_{+_p}, \Bc_{-_p})$ is obtained from $(\Bc_+, \Bc_-)$ by the shift by $p$. 
Note that $L(\Bc_+)$ (resp. $L(\Bc_-)$) coincides with 
$f_{\b_0}^{(c_0)}f_{\b_{-1}}^{(c_{-1})}f_{\b_{-2}}^{(c_{-2})}\cdots$ 
(resp. $\cdots f_{\b_3}^{(c_3)}f_{\b_2}^{(c_2)}f_{\b_1}^{(c_1)}$). 
A similar discussion works for $\ul\BU_q^-$.  We fix $\ul \Bh$ as in (2.5.5).
$\b_k \in \ul\vD^+$ for $k \in \BZ$ is defined similarly to (2.6.1), 
and the root vectors $\ul f_{\b_k} \in \ul \BU_q^-$ are defined as in (2.6.3). 
For $\ul\Bc_{+_p} = (\g_p, \g_{p-1}, \dots) \in \BN^{\BZ_{\le p}}$, 
    $\ul\Bc_{-_p} = (\g_{p+1}, \g_{p+2}, \dots) \in \BN^{\BZ_{> p}}$, 
define $L(\ul\Bc_{+_p}), L(\ul\Bc_{-_p}) \in \ul\BU_q^-$ similarly to (2.6.4).
\par
It is known by [BN, Remark 3.6], for $i \in I_0, \eta \in \ul I_0$,  
\begin{align*}
\tag{2.6.5}
f_{k\d + \a_i} &= T^k_{-\w_i}f_i, \ (k \ge 0), &\quad
f_{k\d - \a_i} &= T^{-k}_{-\w_i}T_if_i, \ (k > 0), \\
\tag{2.6.6} 
\ul f_{k|\eta|\d + \a_{\eta}} &= T^k_{-\wt{\ul\w}_{\eta}}\ul f_{\eta}, \ (k \ge 0), &\quad
\ul f_{k|\eta|\d - \a_{\eta}} &= T^{-k}_{-\wt{\ul\w}_{\eta}}T_{\eta}\ul f_{\eta}, \ (k \ge 0).
\end{align*}
\para{2.7.}
For $i \in I_0, \eta \in \ul I_0, k > 0$, put 
\begin{align*}
\tag{2.7.1}
\wt\psi_{i,k} &= f_{k\d - \a_i}f_i - q^2f_if_{k\d - \a_i}, \\
\tag{2.7.2}
\wt{\ul\psi}_{\eta, k|\eta|} &= \ul f_{k|\eta|\d - \a_{\eta}}\ul f_{\eta}
                 - q^2_{\eta}\ul f_{\eta}\ul f_{k|\eta|\d - \a_{\eta}}.
\end{align*}
It is known that $\wt\psi_{i,k}$ ($i \in I_0, k \in \BZ_{>0}$) 
are mutually commuting, and similarly, 
$\ul\psi_{\eta, k|\eta|}$ ($\eta \in \ul I_0, k \in \BZ_{> 0}$) are 
mutually commuting. 
For each $i \in I_0, k \in \BZ_{>0}$, we define $\wt P_{i,k} \in \BU_q^-$ 
by the following recursive identity;

\begin{equation*}
\tag{2.7.3}
\wt P_{i,k} = \frac{1}{[k]_q}\sum_{s = 1}^kq^{s-k}\wt\psi_{i,s}\wt P_{i,k-s}. 
\end{equation*}  
Similarly, for $\eta \in \ul I_0, k \in \BZ_{> 0}$, we define 
$\wt{\ul P}_{\eta, k|\eta|} \in \ul\BU_q^-$ by
\begin{equation*}
\tag{2.7.4}
\wt{\ul P}_{\eta, k|\eta|} = 
\frac{1}{[k]_{q_{\eta}}}\sum_{s = 1}^kq_{|\eta|}^{s-k}\wt{\ul\psi}_{\eta, s|\eta|}
              \wt{\ul P}_{\eta, (k-s)|\eta|}.
\end{equation*} 
For a fixed $i \in I_0$, regarding $\wt P_{i,k}$ ($k \in \BZ_{>0}$) as 
elementary symmetric functions, we define Schur polynomials by making use of 
the determinant formula; for each partition $\r^{(i)}$, put
\begin{equation*}
\tag{2.7.5}
S_{\r^{(i)}} = \det \bigl(\wt P_{i, \r'_k - k + m}\bigr)_{1 \le k, m \le t}
\end{equation*}
where $(\r'_1, \dots, \r'_t)$ is the dual partition of $\r^{(i)}$. 
For an $|I_0|$-tuple of partitions 
$\Bc_0 = (\r^{(i)})_{i \in I_0}$, we define $S_{\Bc_0}$ by 
\begin{equation*}
\tag{2.7.6}
S_{\Bc_0} = \prod_{i \in I_0}S_{\r^{(i)}}.
\end{equation*}

Similarly, for a fixed $\eta \in \ul I_0$, choose a 
partition $\ul\r^{(\eta)}$, and define a Schur polynomial by 

\begin{equation*}
\tag{2.7.7}
\ul S_{\ul\r^{(\eta)}}  = 
     \det\bigl(\wt{\ul P}_{\eta, (\r'_k - k + m)|\eta|}\bigr)_{1 \le k,m \le t}
\end{equation*}
where $(\r'_1, \dots, \r'_t)$ is the dual partition of $\ul\r^{(\eta)}$. 
For an $\ul I_0$-tuple of partitions $\ul\Bc_0 = (\ul\r^{(\eta)})_{\eta \in \ul I_0}$, 
we define 
\begin{equation*}
\tag{2.7.8}
\ul S_{\ul\Bc_0} = \prod_{\eta \in \ul I_0}\ul S_{\ul \r^{(\eta)}}.
\end{equation*}

\par
We denote by $\SC$ the set of triples $\Bc = (\Bc_+, \Bc_0, \Bc_-)$, 
where $\Bc_+ \in \BN^{\BZ_{\le 0}}, \Bc_- \in \BN^{\BZ_{> 0}}$, 
and $\Bc_0$ is an $I_0$-tuple of partitions. For each $\Bc \in \SC, 
p \in \BZ$, we define $L(\Bc,p) \in \BU_q^-$ by  
\begin{equation*}
\tag{2.7.9}
L(\Bc, p) = \begin{cases}
         L(\Bc_{+_p})\times \bigl(T\iv_{i_{p +1}}T\iv_{i_{p+2}}\cdots T\iv_{i_0}(S_{\Bc_0})\bigr)
               \times L(\Bc_{-_p}), &\quad\text{ if } p \le 0, \\
         L(\Bc_{+_p}) \times \bigl(T_{i_p}\cdots T_{i_2} T_{i_1} (S_{\Bc_0})\bigr) 
               \times L(\Bc_{-p}),  &\quad\text{ if } p > 0.          
               \end{cases}
\end{equation*}
Similarly, we denote by $\ul\SC$ the set of triples 
$\ul\Bc = (\ul\Bc_+, \ul\Bc_0, \ul\Bc_-)$, where $\ul\Bc_+ \in \BN^{\BZ_{\le 0}},
\ul\Bc_- \in \BN^{\BZ_{> 0}}$, and $\ul\Bc_0$ is the set of $\ul I_0$-tuples of partitions.
We define $L(\ul\Bc, p) \in \ul\BU_q^-$ in a similar way as in (2.7.9). 
The following results are proved in [BN].  Note that Lemma 3.39 in [BN]
can be applied to the case where $X$ is simply laced.  

\begin{prop}[{[BN, Prop. 3.16]}]  
$L(\Bc, p) \in {}_{\BA}\BU_q^-$, and $L(\ul\Bc, p) \in {}_{\BA}\ul\BU_q^-$. 
\end{prop}

\begin{prop}[{[BN, Thm 3.13 (i), Lemma 3.39]}]  
We fix $\Bh$ and $p$ as before. 
\begin{enumerate}
\item
For various $\Bc \in \SC$, $L(\Bc, p)$ are almost orthonormal, namely, 
\begin{equation*}
(L(\Bc, p), L(\Bc', p)) \in \d_{\Bc,\Bc'} + q\BZ[[q]] \cap \BQ(q). 
\end{equation*}
In particular, for a fixed $\Bh, p$, 
$\{ L(\Bc, p) \mid \Bc \in \SC\}$ gives a $\BQ(q)$-basis of $\BU_q^-$.   
\par
Similarly, $L(\ul\Bc, p)$ are almost orthonormal, and 
$\{ L(\ul\Bc, p)\mid \ul\Bc \in \ul \SC\}$ gives a $\BQ(q)$-basis of 
$\ul\BU_q^-$.  
\item
$\{ L(\Bc, p) \mid \Bc \in \SC \}$ gives an $\BA$-basis of ${}_{\BA}\BU_q^-$. 
\end{enumerate}
\end{prop}

\para{2.10.}
We first fix $\ul\Bh$ as in (2.5.5), then construct $\Bh$ 
as in (2.5.4) from $\ul\Bh$ by making use of of the relation (2.5.3). 
We also fix $\ul p > 0$, and consider the sequence 
$w_{\ul p} = s_{\eta_{\ul p}}s_{\eta_{\ul p -1}}s_{\eta_{p-2}} 
     \cdots$ in $\ul W  \simeq W^{\s}$. 
Then $w_{\ul p}$ determines an integer $p > 0$ such that $w_{\ul p}$
corresponds to $w_p = s_{i_p}s_{i_{p-1}}s_{i_{p-2}}\cdots$ in $W$.   
For each $s_{\eta_k}$ appearing in $w_{\ul p}$, let $I_k$ be an interval in $\BZ$
such that $s_{\eta_k} = \prod_{j \in I_k}s_{i_j}$ corresponds to a 
subexpression of $w_p$ as above.
Put $F_{\eta_k}(\Bc_{\pm_p}) = \prod_{j \in I_k}f_{i_j}^{(c_j)}$.  
We also define $R_{\eta} = \prod_{j \in \eta}T_j$ for $\eta \in \ul I_0$.
Then $\s$ commutes with $R_{\eta}$.  Note that 
$L(\Bc_{+_p}), L(\Bc_{-_p})$ can be expressed as 
\begin{align*}
\tag{2.10.1}
L(\Bc_{+_p}) &= F_{\eta_{\ul p}}(\Bc_{+_p})R_{\eta_{\ul p}}(F_{\eta_{\ul p -1}}(\Bc_{+_p})) 
           R_{\eta_{\ul p}}R_{\eta_{\ul p-1}}(F_{\eta_{\ul p -2}}(\Bc_{+_p}))\cdots, \\ 
L(\Bc_{-_p}) &= \cdots R_{\eta_{\ul p+1}}\iv R_{\eta_{\ul p+2}}\iv 
                 (F_{\eta_{\ul p+3}}(\Bc_{-_p})) 
      R_{\eta_{\ul p+1}}\iv (F_{\eta_{\ul p+2}}(\Bc_{-_p}))F_{\eta_{\ul p+1}}(\Bc_{-_p}).
\end{align*}    

We have a lemma.
\begin{lem}  
Take $\Bh, p$ as in 2.10. 
\begin{enumerate}
\item
$\s$ permutes the PBW-basis 
$\{ L(\Bc, p)\}$ of $\BU_q^-$, namely, 
$\s(L(\Bc, p)) = L(\Bc',p)$ for some $\Bc' \in \SC$. 
\item
Let $\Bc = (\Bc_+, \Bc_0, \Bc_-) \in \SC$.  Then 
$L(\Bc,p)$ is $\s$-stable if and only if $c_j$ is constant for 
each $j \in I_k$ corresponding to $s_{\eta_k}$ in $w_{\ul p}$, and 
$\r^{(i)}$ is constant on $i \in \eta$ for each $\eta \in \ul I_0$. 
In particular, the set of $\s$-stable PBW-basis in $\BU_q^-$ 
with respect to $\Bh, p$ is in bijection with the set of PBW-basis 
$\{ L(\ul\Bc, \ul p)\}$ in $\ul\BU_q^-$ if $\Bh,p$ are obtained from 
$\ul\Bh, \ul p$.  
\end{enumerate}
\end{lem}

\begin{proof}
By (2.10.1), we have $\s(L(\Bc_{+_p})) = L(\Bc'_{+_p})$, 
$\s(L(\Bc_{-_p})) = L(\Bc'_{-_p})$ for some 
$\Bc'_{+_p} \in \BN^{\BZ_{\le p}}, \Bc'_{-_p} \in \BN^{\BZ_{>p}}$. 
On the other hand, since
$\s(f_{k\d \pm \a_i}) = f_{k\d \pm \a_{\s(i)}}$ for $i \in I_0, k > 0$ 
by (2.6.5), 
we have $\s(\wt\psi_{i,k}) = \wt\psi_{\s(i),k}$, and so
$\s(\wt P_{i,k}) = \wt P_{\s(i),k}$.  This implies that 
$\s(S_{\r^{(i)}}) = S_{\r^{(\s(i))}}$ for each $i \in I_0$.   
We see that $\s(S_{\Bc_0}) = S_{\Bc_0'}$ for some $I_0$-tuple of 
partitions $\Bc_0'$.  
Thus we obtain (i).  (ii) follows from (i). 
\end{proof}

\para{2.12.}
We apply the discussion in 1.5 to the affine case,  
and we can define a homomorphism $\pi : {}_{\BA'}\BU_q^{-,\s} \to \BV_q$. 
For any $\eta \in \ul I$, and $a \in \BN$,  we define 
$\wt f^{(a)}_{\eta} = \prod_{i \in I}f_i^{(a)}$, and put 
$g^{(a)}_{\eta} = \pi(\wt f^{(a)}_{\eta})$ as in 1.9.  
Then Proposition 1.10 still holds for the affine case, and we can define 
an algebra homomorphism $\Phi : {}_{\BA'}\ul \BU_q^- \to \BV_q$ of $\BA'$-algebras. 
Assume that $\Bh, p$ are obtained from $\ul\Bh, \ul p$ as in 2.10.  We denote by 
$\SX_{\Bh, p}$ the set of PBW-basis 
$\{ L(\Bc, p) \mid \Bc \in \SC\}$ of $\BU_q^-$, 
and $\SX^{\s}_{\Bh,p}$ the subset of 
$\SX_{\Bh,p}$ consisting of $\s$-stable PBW-basis. 
Similarly, we denote by $\ul\SX_{\ul\Bh, \ul p}$ the set of PBW-basis 
$\{ L(\ul\Bc, \ul p) \mid \ul\Bc \in \ul\SC\}$ of $\ul\BU_q^-$. 
By Lemma 2.11 (ii), we have a natural bijection 
$\SX^{\s}_{\Bh,p} \simeq \ul\SX_{\ul\Bh, \ul p}$, by $L(\Bc, p) \lra L(\ul\Bc, \ul p)$. 
We put $E(\ul\Bc, \ul p) = \pi(L(\Bc, p))$ under this correspondence. 
Then by Lemma 2.11 (i), and by Proposition 2.9 (see the discussion in 1.12), 
we see that $\{ E(\ul\Bc, \ul p)\}$ gives rise to an $\BA'$-basis of $\BV_q$. 
\par
Assume that $L(\Bc, p) \in \SX^{\s}_{\Bh,p}$ corresponds to 
$L(\ul\Bc, \ul p) \in \ul\SX_{\ul\Bh, \ul p}$ with 
$\Bc = (\Bc_+, \Bc_0, \Bc_-)$, $\ul\Bc = (\ul\Bc_+, \ul\Bc_0, \ul\Bc_-)$.
We consider $L(\Bc_{+_p}), L(\Bc_{-_p}) \in \BU_q^{-,\s}$ and 
$L(\ul\Bc_{+_{\ul p}}), L(\ul\Bc_{-_{\ul p}}) \in \ul\BU_q^-$. 
The following result can be proved in a similar way as in Theorem  1.14 (i).

\begin{prop}  
$\Phi(L(\ul\Bc_{+_{\ul p}})) = \pi(L(\Bc_{+_p}))$ and 
$\Phi(L(\ul\Bc_{-_{\ul p}})) = \pi(L(\Bc_{-_p}))$. 
\end{prop} 
 
\para{2.14.}
Let $\Bc_0 = (\r^{(i)})_{i \in I_0}$ be an $I_0$-tuple of partitions appearing 
in $\Bc$, and 
$\ul\Bc_0 = (\ul\r^{(\eta)})_{\eta \in \ul I_0}$ be an $\ul I_0$tuple of 
partitions appearing in $\ul\Bc$ as in 2.7.  
We have $\r^{(i)} = \ul\r^{(\eta)}$ if $i \in \eta$ for each $\eta \in \ul I_0$.     
Then $S_{\Bc_0} \in \BU_q^{-,\s}$, and we consider $\pi(S_{\Bc_0}) \in \BV_q$. 
On the other hand, we can consider $\ul S_{\ul\Bc_0} \in \ul\BU_q^-$. 
We show a lemma.

\begin{lem}  
$\Phi(\ul S_{\ul \Bc_0}) = \pi(S_{\Bc_0})$. 
\end{lem}

\begin{proof}
Take $i \in I_0$ such that $i \in \eta$. We consider 
$\prod_{i \in \eta}f_{k\d + \a_i} \in \BU_q^{-,\s}$ and 
$\ul f_{k|\eta|\d + \a_{\eta}} \in \ul\BU_q^-$, and similar elements 
obtained by replacing $\a_i$ by $-\a_i$, $\a_{\eta}$ by $-\a_{\eta}$. 
By applying Proposition 2.13 for the case where $p = 0$, we have
\begin{align*}
\tag{2.15.1}
\Phi(\ul f_{k|\eta|\d + \a_{\eta}}) = \pi(\prod_{i \in \eta}f_{k\d + \a_i}),  \qquad 
\Phi(\ul f_{k|\eta|\d - \a_{\eta}}) = \pi(\prod_{i \in \eta}f_{k\d - \a_i}).
\end{align*}  
Next we show, for $\eta \in \ul I_0, k > 0$,  that
\begin{equation*}
\tag{2.15.2}
\Phi(\wt{\ul\psi}_{\eta, k|\eta|}) = \pi(\prod_{i \in \eta}\wt\psi_{i,k}).
\end{equation*}

It is known by [B, BCP] that $T_{\w_i}(f_{k\d \pm \a_j}) = f_{k\d \pm \a_j}$ 
for $i \ne j, k\ge 0$. Hence 
if $(\a_i, \a_j) = 0$, we have
\begin{equation*}
\tag{2.15.3}
f_jf_{k\d - \a_i} = f_jT^{-k}_{-\w_i}T_i(f_i) = T^{-k}_{\w_i}T_i(f_jf_i) 
                  = T^{-k}_{\w_i}T_i(f_if_j) = f_{k\d - \a_i}f_j
\end{equation*} 
by (2.6.5). Again by using (2.6.5) we have
\begin{equation*}
\tag{2.15.4}
f_{k\d - \a_i}f_{k\d - \a_j} = f_{k\d - \a_j}f_{k\d - \a_i}.
\end{equation*}
\par  
In the case where $|\eta| = 1$, (2.15.2) immediately follows from (2.15.1). 
We assume that $|\eta| = 2$, and put $\eta = \{ i, j\}$.  Then 
by using commutation relations (2.15.3), (2.15.4), we have 
\begin{align*}
\wt\psi_{i, k}\wt\psi_{j, k} &= 
          (f_{k\d - \a_i}f_i - q^2f_if_{k\d - \a_i})
          (f_{k\d - \a_j}f_j - q^2f_jf_{k\d - \a_j})  \\
 &= f_{k\d - \a_i}f_{k\d - \a_j}f_if_j + q^4f_if_jf_{k\d - \a_i}f_{k\d - \a_j} -  q^2Z , 
\end{align*}
where 
\begin{align*}
Z &= f_{k\d - \a_i}f_if_jf_{k\d - \a_j} + f_if_{k\d - \a_i}f_{k\d - \a_j}f_j \\
  &= f_jf_{k\d - \a_i}f_{k\d - \a_j}f_i + f_if_{k\d - \a_i}f_{k\d - \a_j}f_j \\
  &= f_jf_{k\d - \a_j}f_{k\d - \a_i}f_i + \s(f_jf_{k\d - \a_j}f_{k\d - \a_i}f_i).  
\end{align*}
Since $Z \in J$, we have
\begin{equation*}
\pi(\wt\psi_{i,k}\wt\psi_{j,k}) 
    = \pi (f_{k\d - \a_i}f_{k\d - \a_j}f_if_j 
      - q^2_{\eta}f_if_jf_{k\d - \a_i}f_{k\d - \a_j}).  
\end{equation*}
Now (2.15.2) follows from (2.15.1).  The proof for the case $|\eta| = 3$ 
is similar. Thus (2.15.2) is proved.
\par
Since $\wt\psi_{i,k}$ and $\wt\psi_{j, \ell}$ commute for any pair, 
$\wt\psi_{i,k}$ commutes with $\wt P_{j,\ell}$ for any pair $i,j, k,\ell$. 
Then by a similar argument as in the proof of (2.15.2), 
for each $\eta \in \ul I_0$ we have
\begin{equation*}
\tag{2.15.5}
\Phi(\wt {\ul P}_{\eta, k}) = \pi\bigl(\prod_{i \in \eta}\wt P_{i,k}\bigr).
\end{equation*}
(Note that $([k]_q)^{|\eta|} = [k]_{q_{\eta}}$ in $\BA'$. ) 
\par
Since $\wt P_{i,k}$ are commuting for any pair $i,k$, (2.15.5) implies, 
by a similar argument as above, that
\begin{equation*}
\tag{2.15.6}
\Phi(\ul S_{\r^{(\eta)}}) = \pi(\prod_{i \in \eta}S_{\r^{(i)}})
\end{equation*}
for any $\eta \in \ul I_0$. Lemma 2.15 follows from this. 
\end{proof}

The following result is an analogue of Theorem 1.14 and Proposition 1.20.
\begin{thm}  
\begin{enumerate}
\item 
For any $\ul\Bc  \in \ul\SC$, we have $\Phi(L(\ul\Bc, \ul 0)) = E(\ul\Bc, \ul 0)$. 
\item 
PBW-basis $\{ L(\ul\Bc, \ul 0) \mid \ul\Bc \in \ul\SC\}$ 
gives an $\BA'$-basis of ${}_{\BA'}\ul\BU_q^-$. 
\item 
$\Phi$ gives an isomorphism ${}_{\BA'}\ul\BU_q^- \isom \BV_q$. 
\end{enumerate}
\end{thm}

\begin{proof}
(i) follows from Proposition 2.13 and Lemma 2.15.
By Proposition 2.8, (the image of ) $L(\ul\Bc, \ul 0)$ is contained in ${}_{\BA'}\ul\BU_q^-$.
Hence the map $\Phi : {}_{\BA'}\ul\BU_q^- \to \BV_q$ is surjective. 
As in the proof of Theorem 1.14, $\Phi$ can be extended to the map 
${}_{\BF(q)}\ul\BU^-_q \to {}_{\BF(q)}\BV_q$, which gives an isomorphism of $\BF(q)$-algebras.
Let ${}_{\BA'}\wt{\ul\BU}_q^-$ be the $\BA'$-submodule of ${}_{\BF(q)}\ul\BU_q^-$ 
spanned by $L(\ul\Bc,0)$.
Then $\Phi$ gives an isomorphism 
${}_{\BA'}\wt{\ul\BU}_q^- \simeq \BV_q$ of $\BA'$-modules. 
In particular, ${}_{\BA'}\wt{\ul\BU}_q^-$ is an algebra over $\BA'$.
We note that 
\begin{equation*}
\tag{2.16.1}
{}_{\BA'}\wt{\ul\BU}_q^- = {}_{\BA'}\ul\BU_q^-.
\end{equation*}

In fact, ${}_{\BA'}\wt{\ul\BU}_q^- \subset {}_{\BA'}\ul\BU_q^-$ by Proposition 2.8.
Since $\{ E(\ul\Bc, \ul p) \mid \ul\Bc \in \ul\SC \}$ is an $\BA'$-basis of 
$\BV_q$, $\{ \Phi\iv(E(\ul\Bc, \ul p)) \mid \ul\Bc \in \ul\SC\}$ 
gives an $\BA'$-basis of ${}_{\BA'}\wt{\ul\BU}_q^-$ for any $\ul p$.. 
Hence by (2.10.1), ${}_{\BA'}\wt{\ul\BU}_q^-$ is invariant under 
the left multiplication by $\ul f_{\eta_{\ul p}}^{(k)}$. 
By (2.5.7), for any $\eta \in \ul I$, there exists $\ul p$ such that
$\eta = \eta_{\ul p}$. Thus ${}_{\BA'}\wt{\ul\BU}_q^-$ is invariant under 
the left multiplication by any $\ul f^{(k)}_{\eta}$, and (2.16.1) follows.
Now (ii) and (iii) follows from (2.16.1).  The theorem is proved.
\end{proof}

\begin{cor}  
For any $\ul p \in \BZ$, 
the PBW-basis $\{L(\ul\Bc, \ul p) \mid \ul\Bc \in \ul\SC \}$ 
gives an $\BA$-basis of ${}_{\BA}\ul\BU_q^-$. 
\end{cor}

\begin{proof}
By a similar argument as in the proof of Corollary 1.21, we see that
$\{ L(\ul\Bc, \ul 0) \mid \ul\Bc \in \ul\SC \}$ 
gives an $\BA$-basis of ${}_{\BA}\ul\BU_q^-$ thanks to 
Theorem 2.16.  Then by [BN, Lemma 3.39], $\{ L(\ul\Bc, \ul p)\}$ gives an  
$\BA$-basis of ${}_{\BA}\ul\BU_q^-$.  The corollary is proved. 
\end{proof}

\remark{2.18.}
In the case where $\Fg$ is a simply laced affine algebra, 
the fact that $\{ L(\Bc, p) \mid \Bc \in \SC\}$ 
gives an $\BA$-basis of ${}_{\BA}\BU_q^-$ (Proposition 2.9 (ii)) was known by 
[BCP] for $p = 0$, and was proved by [BN] for arbitrary $p$.  
Corollary 2.17 is a generalization of this fact to the case of twisted 
affine algebras. 
Once this is done, one can define the (signed) canonical basis 
$b(\Bc, p)$ parametrized by $L(\Bc, p)$ as in (1.22.2).
The basis $\{ b(\Bc,p) \mid \Bc \in \SC\}$ 
is independent of the choice of $\Bh$ and $p$, up to $\pm 1$.
In [BN], in the simply laced case, this ambiguity of the sign was 
removed by using the theory of extremal weight modules due to [K2].
It is likely that our result makes it possible to extend their
results to the case of twisted affine Lie algebras.      

\par\bigskip
\section{The proof of Proposition 1.10}

\para{3.1.}
In this and next section we write $[a]_{q^i}$ as $[a]_i$ for any $i \in \BZ$. 
Thus $[a]_q = [a]_1$ and $[a]_{q_{\eta}} = [a]_{|\eta|}$ since 
$(\a_{\eta}, \a_{\eta})_1/2 = |\eta|$. 
${}_{\BA'}\ul\BU_q^-$ is the $\BA'$-algebra with generators 
$ \ul f_{\eta}^{(a)}$  $(\eta \in \ul I, a \in \BN)$ with fundamental relations 
\begin{align*}
\tag{3.1.1}
&\sum_{k= 0}^{1 - a_{\eta\eta'}}(-1)^k
       \ul f_{\eta}^{(k)}\ul f_{\eta'}\ul f_{\eta}^{(1 - a_{\eta\eta'} - k)} = 0,
          \qquad (\eta \ne \eta'), \\ 
\tag{3.1.2}
&[a]^!_{|\eta|}\ul f_{\eta}^{(a)} = \ul f_{\eta}^a, \qquad (a \in \BN),
\end{align*} 
where $A = (a_{\eta\eta'})$ is the Cartan matrix of $\ul X$. 
In order to prove Proposition 1.10, it is enough to show that 
$g_{\eta}^{(a)}$ satisfies a similar relations as above, namely, 
\begin{align*}
\tag{3.1.3}
&\sum_{k= 0}^{1 - a_{\eta\eta'}}(-1)^k
       g_{\eta}^{(k)}g_{\eta'}g_{\eta}^{(1 - a_{\eta\eta'} - k)} = 0,
          \qquad (\eta \ne \eta'), \\ 
\tag{3.1.4}
&[a]^!_{|\eta|}g_{\eta}^{(a)} = g_{\eta}^a, \qquad (a \in \BN).
\end{align*} 
First we show (3.1.4). 
We have 
\begin{align*}
\wt f_{\eta}^{(a)} = \prod_{i \in \eta}f_i^{(a)} 
        = ([a]^!_1)^{-|\eta|}\prod_{i \in \eta}f_i^a = ([a]^!_1)^{-|\eta|}\wt f_{\eta}^a. 
\end{align*}
Since $|\eta| = 1$ or $\ve$, we have 
$([a]^!_1)^{|\eta|} = [a]^!_{|\eta|}$ in $\BA' = \BF[q, q\iv]$ 
with $\BF = \BZ/\ve\BZ$. Thus (3.1.4) follows. 
\par
For the proof of (3.1.3), 
we may assume that $\ul X$ is of rank 2. 
Here we change the notation from 1.3, and consider $\ul I = \{ \ul 1, \ul 2 \}$ with 
Cartan matrix 
\begin{equation*}
A = \begin{pmatrix}
          2  &  0   \\
          0  &  2
     \end{pmatrix}
       \quad\text{ or } \quad
    \begin{pmatrix}
           2  &  a \\
           -1 &  2
    \end{pmatrix}
\end{equation*} 
where $\ul X$ is of type $A_1 \times A_1$ in the first case, 
and $a = -1, -2, -3$ according to the cases 
$\ul X$ is of type $A_2, B_2, G_2$. 
\par
Assume that $\ul X$ is of type $A_1 \times A_1$.
In this case, we have
$(\a_i, \a_j) = 0$ for any $i, j \in I$ 
such that $i \ne j$.  
It is easily seen that 
$g_{\ul 1}g_{\ul 2} = g_{\ul 2}g_{\ul 1}$, which coincides with the relation 
(3.1.3).  Thus (3.1.3) holds.

\para{3.2.} 
Assume that $\ul X$ is of type $A_2$. We have two possibilities for $I$, 
$\ul i = \{ i \}$ or $\ul i = \{ i,i'\}$ for $i = 1,2$. 
In the former case, (3.1.3) clearly holds.  So we may assume that 
$I = \{ 1,2,1',2'\}$ with $\ul 1 = \{ 1,1'\}, \ul 2 = \{ 2,2'\}$, 
where $(\a_i, \a_j) =  -1$ for $\{ i, j\} = \{ 1,2\}$ or $\{ 1',2'\}$, 
and is equal to zero for other cases.
We have $g_{\ul 1} = \pi(f_1f_{1'})$ and $g_{\ul 2} = \pi(f_2f_{2'})$.  
The relation (3.1.3) is given by 
\begin{align*}
\tag{3.2.1}
g_{\ul 1}g_{\ul 2}^{(2)} - g_{\ul 2}g_{\ul 1}g_{\ul 2} + g_{\ul 2}^{(2)}g_{\ul 1} = 0.
\end{align*}   
By (3.1.4), this is equivalent to 
\begin{equation*}
\tag{3.2.2}
g_{\ul 1}g_{\ul 2}^2 - (q^2 + q^{-2})g_{\ul 2}g_{\ul 1}g_{\ul 2} + g_{\ul 2}^2g_{\ul 1} = 0.
\end{equation*}
We show (3.2.2). 
It follows from the Serre relations for $A_2$, we have 
\begin{align*}
\tag{3.2.3}
&f_1f_2^2 - (q + q\iv)f_2f_1f_2 + f^2_2f_1 = 0, \\
&f_2f_1^2 - (q + q\iv)f_1f_2f_1 + f^2_1f_2 = 0, \\
\end{align*}
and formulas obtained form (3.2.3) by replacing $f_1, f_2$ by $f_{1'}, f_{2'}$. 
By multiplying these two formulas, and by using the commutation relations
$f_if_j = f_jf_i$ unless $\{ i,j\} = \{ 1, 2\}$ nor $\{ 1', 2'\}$, we have
\begin{align*}
(f_1f_{1'})(f_2f_{2'})^2 + (q + q\iv)^2 (f_2f_{2'})(f_1f_{1'})(f_2f_{2'}) 
      + (f_2f_{2'})^2(f_1f_{1'}) + Z = 0,  
\end{align*}
where 
\begin{equation*}
Z = -(q + q\iv)\bigl( f_2f_{2'}^2f_1f_{1'}f_2 + f_{2'}f_2^2f_{1'}f_1f_{2'} \bigr)
\end{equation*}
Since $\ve = 2$, and $Z \in J$,  we obtain (3.2.2). 
Thus (3.1.3) is verified for $\ul X$ of type $A_2$. 

\para{3.3.}
Assume that $\ul X$ is of type $B_2$ and $X$ is of type $A_3$. 
We have $I = \{ 1, 2, 2'\}$, $\ul I = \{ \ul 1, \ul 2\}$ with 
$\ul 1 = \{ 1 \}$ and $\ul 2 = \{ 2,2'\}$, where 
$(\a_i,\a_j) = -1$ if $\{ i, j\} = \{ 1, 2\}$ or $\{ 1,2'\}$ and is equal to 
zero for all other $i \ne j$. By (3.1.4), (3.1.3) is equivalent to the formulas
\begin{align*}
\tag{3.3.1}
&g_{\ul 1}g_{\ul 2}^2 - (q^2 + q^{-2})g_{\ul 2}g_{\ul 1}g_{\ul 2}
                         + g_{\ul 2}^2g_{\ul 1} = 0, \\
\tag{3.3.2}
&g_{\ul 2}g_1^3 - [3]_1g_{\ul 1}g_{\ul 2}g_{\ul 1}^2 + [3]_1g_{\ul 1}^2g_{\ul 2}g_{\ul 1}
               - g_{\ul 1}^3g_{\ul 2} = 0.               
\end{align*}  

We show (3.3.1).  Here $\BU_q^-$ satisfies the formulas (3.2.3) and the formulas 
obtained from (3.2.3) by replacing $f_1, f_2$ by $f_1, f_{2'}$. 
By multiplying $f_{2'}^2$ from the right on (3.2.3) for $f_1f_2^2$, we have 
\begin{equation*}
\tag{3.3.3}
f_1f_2^2f_{2'}^2 - (q + q\iv)f_2f_1f_{2'}^2f_2 + f_2^2f_1f_{2'}^2 = 0. 
\end{equation*}
Here by applying (3.2.3) for $f_1f_{2'}^2$, we have
\begin{align*}
f_2(f_1f_{2'}^2)f_2 &= (q + q\iv)(f_2f_{2'})f_1(f_2f_{2'}) - f_2f_{2'}^2f_1f_2,  \\ 
f_2^2(f_1f_{2'}^2) &= (q + q\iv)f_2^2f_{2'}f_1f_{2'} - (f_2f_{2'})^2f_1.
\end{align*}
Substituting these formulas into (3.3.3), we have 
\begin{equation*}
f_1(f_2f_{2'})^2 - (q + q\iv)^2(f_2f_{2'})f_1(f_2f_{2'}) - (f_2f_{2'})^2f_1 + Z = 0, 
\end{equation*}
where 
\begin{equation*}
Z = (q + q\iv) \bigl( f_{2'}^2f_2f_1f_2 + f_2^2f_{2'}f_1f_{2'}\bigr).
\end{equation*}
Since $\d = 2$, $Z \in J$, we obtain (3.3.1). 
\par
Next we show (3.3.2). 
First note the following equality. By using (3.2.3) for $f_1f_2f_1$ 
and for $f_{2'}f_1^2$, we have 
\begin{align*}
\tag{3.3.4}
(q + q\iv)f_1f_{2'}(f_1f_2f_1) &= f_1f_{2'}(f_2f_1^2 + f_1^2f_2)  \\
                &= f_1(f_2f_{2'})f_1^2 + f_1(f_{2'}f_1^2)f_2 \\
                &= f_1(f_2f_{2'})f_1^2 + (q + q\iv)f_1^2f_{2'}f_1f_2 - f_1^3f_{2'}f_2. 
\end{align*}

Here by applying (3.2.3) for $f_2f_1^2$ and for $f_{2'}f_1^2$ twice, 
we have 
\begin{align*}
f_{2'}f_1(f_2f_1^2) &= f_{2'}f_1\bigl((q + q\iv)f_1f_2f_1 - f_1^2f_2\bigr)  \\
                    &= f_{2'}f_1^2\bigl((q + q\iv)f_2f_1 - f_1f_2\bigr) \\
                    &= \bigl((q + q\iv)f_1f_{2'}f_1 - f_1^2f_{2'}\bigr)
                       \bigl((q + q\iv)f_2f_1 - f_1f_2\bigr) \\
                    &=  (q + q\iv)^2f_1f_{2'}f_1f_2f_1 - (q + q\iv)f_1(f_{2'}f_1^2)f_2
                       - (q + q\iv)f_1^2f_{2'}f_2f_1 + f_1^2f_{2'}f_1f_2 \\
                    &= (q + q\iv)^2f_1f_{2'}f_1f_2f_1 
                       - \bigl((q + q\iv)^2 - 1\bigr)f_1^2f_{2'}f_1f_2 \\
                    &\phantom{***} 
                        - (q + q\iv)f_1^2f_2f_{2'}f_1 + (q + q\iv)f_1^3f_2f_{2'}.  
\end{align*}
Substituting (3.3.4) into the last equality, we obtain 
\begin{align*}
\tag{3.3.5}
f_{2'}f_1f_2f_1^2 &= (q + q\iv)f_1(f_2f_{2'})f_1^2 - (q + q\iv)f_1^2(f_2f_{2'})f_1 
+ f_1^2f_{2'}f_1f_2. 
\end{align*}
On the other hand, by applying (3.2.3) for $f_{2'}f_1^2$, we have
\begin{align*}
\tag{3.3.6}
(f_{2'}f_1^2)f_2f_1  &= (q + q\iv)f_1f_{2'}f_1f_2f_1 - f_1^2f_{2'}f_2f_1  \\   
                     &= f_1(f_2f_{2'})f_1^2 + (q + q\iv)f_1^2f_{2'}f_1f_2 - f_1^3f_2f_{2'}
                    - f_1^2f_{2'}f_2f_1.
\end{align*}
The second identity is obtained by substituting (3.3.4) into the first identity. 
Now by applying (3.2.3) for $f_2f_1^2$ we have 
\begin{align*}
f_{2'}f_2f_1^3 &= f_{2'}(f_2f_1^2)f_1 
                = (q + q\iv)f_{2'}f_1f_2f_1^2 - f_{2'}f_1^2f_2f_1.
\end{align*} 
By substituting (3.3.5) and (3.3.6) into the last formula, we have
\begin{align*}
\tag{3.3.7}
(f_2f_{2'})f_1^3 &= (q^2 + 1 + q^{-2}) f_1(f_2f_{2'})f_1^2 
                - (q^2 + 1  + q^{-2})f_1^2(f_2f_{2'})f_1
                + f_1^3(f_2f_{2'}). 
\end{align*}
Since $[3]_1 = q^2 + 1 + q^{-2}$, by applying $\pi$, we obtain (3.3.2).
Note that the formula (3.3.7) is obtained without appealing modulo 2. 
Thus (3.1.3) is verified for $\ul X$ of type $B_2$. 

\para{3.4.}
Assume that $\ul X$ is of type $G_2$ and $X$ is of type $D_4$. 
We have $I = \{ 1,2,2',2''\}$, $\ul I = \{ \ul 1, \ul 2\}$ with 
$\ul 1 = \{ 1 \}$ and $\ul 2 = \{ 2,2',2''\}$, where $(\a_i, \a_j) = -1$ 
if $\{ i, j\} = \{ 1, 2\}, \{1, 2'\}$ or $\{ 1, 2''\}$, and is equal to zero 
for all other $i \ne j$. 
By (3.1.4), (3.1.3) is equivalent to the formulas

\begin{align*}
\tag{3.4.1}
&g_{\ul 1}g_{\ul 2}^2 - (q^3 + q^{-3})g_{\ul 2}g_{\ul 1}g_{\ul 2} + g_{\ul 2}^2g_{\ul 1} = 0, \\
\tag{3.4.2}
&g_{\ul 2}g_{\ul 1}^4 - [4]_1g_{\ul 1}g_{\ul 2}g_{\ul 1}^3 
                      + \begin{bmatrix}
                          4  \\
                          2
                        \end{bmatrix}_1 g_{\ul 1}^2g_{\ul 2}g_{\ul 1}^2
                      - [4]_1g_{\ul 1}^3g_{\ul 2}g_{\ul 1} + g_{\ul 1}^4g_{\ul 2} = 0, 
\end{align*} 
where $[4]_1 = q^3 + q + q\iv + q^{-3}$ and 
   $\begin{bmatrix}
      4 \\
      2
    \end{bmatrix}_1 = q^4 + q^2 + 2 + q^{-2} + q^{-4}$. 
We show (3.4.1). 
Here $\BU_q^-$ satisfies the formulas (3.2.3) and the formulas 
obtained from (3.2.3) by replacing $f_1, f_2$ by $f_1, f_{2'}$ or $f_1, f_{2''}$. 
By multiplying $f_{2'}^2f_{2''}^2$ from the right on (3.2.3) for $f_1f_2^2$, we have 
\begin{equation*}
\tag{3.4.3}
f_1f_2^2f_{2'}^2f_{2''}^2 - (q + q\iv)f_2f_1f_{2'}^2f_{2''}^2f_2 
         + f_2^2f_1f_{2'}^2f_{2''}^2 = 0. 
\end{equation*}
Concerning the middle term, by applying (3.2.3) for $f_1f_{2'}^2$, 
then for $f_1f_{2''}^2$, we have

\begin{align*}
\tag{3.4.4}
f_2(f_1f_{2'}^2)f_{2''}^2f_2 &=  (q + q\iv)f_2f_{2'}(f_1f_{2''}^2)f_{2'}f_2 
                      - f_2f_{2'}^2f_1f_{2''}^2f_2 \\      
     &=  (q + q\iv)^2f_2f_{2'}f_{2''}f_1f_2f_{2'}f_{2''} 
            - (q + q\iv)f_2f_{2'}f_{2''}^2f_1f_2f_{2'}  \\ 
           &\phantom {*****} - f_2f_{2'}^2f_1f_2f_{2''}^2.
\end{align*}

Concerning the third term, by applying (3.2.3) for $f_1f_{2''}^2$, then 
for $f_1f_{2'}^2$, and finally for $f_2^2f_1$, we have
\begin{align*}
f_2^2(f_1f_{2''}^2)f_{2'}^2 &= (q + q\iv)f_2^2f_{2''}f_1f_{2''}f_{2'}^2 
                                 - f_2^2f_{2''}^2(f_1f_{2'}^2)  \\
         &= (q + q\iv)f_2^2f_{2''}f_1f_{2''}f_{2'}^2 
              - (q + q\iv)f_{2'}f_{2''}^2(f_2^2f_1)f_{2'} + f_2^2f_{2''}^2f_{2'}^2f_1 \\
         &= (q + q\iv)f_2^2f_{2''}f_1f_{2''}f_{2'}^2 
    - (q + q\iv)^2f_{2'}f_{2''}^2f_2f_1f_2f_{2'} \\
    &\phantom{*****} 
        + (q + q\iv)f_{2'}f_{2''}^2f_1f_2^2f_{2'} + f_2^2f_{2''}^2f_{2'}^2f_1.
\end{align*}
It follows that 
\begin{equation*}
f_1(f_2f_{2'}f_{2''})^2 - (q + q\iv)^3(f_2f_{2'}f_{2''})f_1(f_2f_{2'}f_{2''}) 
           + (f_2f_{2'}f_{2''})^2f_1 + Z = 0,  
\end{equation*}
where 
\begin{equation*}
Z = (q + q\iv)\bigl( f_2f_{2'}^2f_1f_{2''}^2f_2 + f_{2''}f_2^2f_1f_{2'}^2f_{2''}
                    + f_{2'}f_{2''}^2f_1f_2^2f_{2'}\bigr).
\end{equation*}
Since $\ve = 3$, $Z \in J$, we obtain (3.4.1). 

\para{3.5.}
It remains to prove (3.4.2).  
We shall prove the following formula in ${}_{\BA'}\BU_q^-$.
\begin{align*}
\tag{3.5.1}
(f_2f_{2'}f_{2''})f_1^4 &- [4]_1f_1(f_2f_{2'}f_{2''})f_1^3 + 
       \begin{bmatrix}
            4  \\
            2
       \end{bmatrix}_1f_1^2(f_2f_{2'}f_{2''})f_1^2  \\ 
            &- [4]_1f_1^3(f_2f_{2'}f_{2''})f_1 + f_1^4(f_2f_{2'}f_{2''}) \equiv 0  \mod J.
\end{align*}
Clearly (3.5.1) will imply (3.4.2).  
The proof of (3.5.1) by the direct computation 
as in the case of $B_2$ seems to be difficult. 
Instead, we will prove 
(3.5.1) by making use of PBW-basis of $\BU_q^-$. 
\par
Let $\Bh = (i_1, \dots, i_{\nu})$ be a sequence associated to 
the longest element $w_0$ of $W$.  Here $W$ is of type $D_4$, and 
$\nu = 12$.  We choose $\Bh$ as
\begin{equation*}
\tag{3.5.2}
\Bh = (2, 2', 2'', 1, 2, 2', 2'', 1, 2, 2', 2'',1). 
\end{equation*}

We define $\b_k = s_{i_1}\cdots s_{i_{k-1}}(\a_{i_k})$ for $k = 1, \dots, \nu = 12$. 
Then the set $\vD^+$ of positive roots is given as 
\begin{align*}
\tag{3.5.3}
\vD^+ &=\{ \b_1, \dots, \b_{12}\} \\
      &= \{ 2, 2', 2'', 122'2'', 12'2'', 122'', 122', 
              1122'2'', 12, 12', 12'', 1\}, 
\end{align*}
where we use the notation for positive roots such as $12 \lra \a_1 + \a_2$, 
$12'2'' \lra \a_1 + \a_{2'} + \a_{2''}$, etc. 
For $k = 1, \dots, \nu$, the root vector $f_{\b_k}^{(c)}$ is defined by 
$f_{\b_k}^{(c)} = T_{i_1}\cdots T_{i_{k-1}}(f_{i_k}^{(c)})$.  
Then PBW-basis of $\BU_q^-$ is given as $\{ L(\Bc, \Bh) \mid \Bc \in \BN^{12} \}$, 
where for $\Bc = (c_1, \dots, c_{12})$, 
\begin{equation*}
L(\Bc,\Bh) = f_2^{(c_1)}f_{2'}^{(c_2)}f_{2''}^{(c_3)}
        f_{122'2''}^{(c_4)}f_{12'2''}^{(c_5)}f_{122''}^{(c_6)}f_{122'}^{(c_7)}
          f_{1122'2''}^{(c_8)}f_{12}^{(c_9)}f_{12'}^{(c_{10})}f_{12''}^{(c_{11})}
               f_1^{(c_{12})}.
\end{equation*}
 
We use the following commutation relations,

\begin{align*}
\tag{3.5.4}
f_{12} &= f_1f_2 - qf_2f_1  \quad (\text{similarly for } \ f_{12'}, f_{12''}), \\
f_{122'} &= f_{12}f_{2'} - qf_{2'}f_{12} = f_{12'}f_2 - qf_2f_{12'}, 
         \quad (\text{similarly for } \ f_{12'2''}, f_{122''}) \\
f_{122'2''} &= f_{122'}f_{2''} - qf_{2''}f_{122'} = f_{12'2''}f_2 - qf_2f_{12'2''}
             = f_{122''}f_{2'} - qf_{2'}f_{122''}, \\
f_{1122'2''} &= f_{12''}f_{122'} - qf_{122'}f_{12''}
            = f_{12}f_{12'2''} - qf_{12'2''}f_{12}
            = f_{12'}f_{122''} - qf_{122''}f_{12'}, 
\end{align*}

The following formulas are obtained by applying 
the commutation formula of Levendorskii and Soibelman [LS]. 
\begin{align*}
f_{122'2''}f_{2} &= q\iv f_{2}f_{122'2''}, \quad (\text{similarly for }
                       \ f_{122'2''}f_{2'}, f_{122'2''}f_{2''}), \\
f_{12'2''}f_{122'2''} &= q\iv f_{122'2''}f_{12'2''}, 
        \quad(\text{similarly for } f_{122'}f_{122'2''}, f_{122''}f_{122'2''}), \\
f_{122''}f_{12'2''} &= f_{12'2''}f_{122''}, 
        \quad(\text{similarly for } \ f_{122'}f_{122''}, f_{122'}f_{12'2''}), \\
f_{1122'2''}f_{122'} &= q\iv f_{122'}f_{1122'2''}, 
        \quad(\text{similarly for } f_{1122'2''}f_{122''}, f_{1122'2''}f_{12'2''}), \\
f_{12}f_{1122'2''} &= q\iv f_{1122'2''}f_{12}, 
        \quad(\text{similarly for } f_{12'}f_{1122'2''}, f_{12''}f_{1122'2''}), \\
f_{12'}f_{12} &= f_{12}f_{12'}, 
        \quad(\text{similarly for } f_{12''}f_{12'},  f_{12''}f_{12}), \\  
f_1f_{12} &= q\iv f_{12}f_1, 
         \quad(\text{similarly for } f_1f_{12'}, f_1f_{12''}). 
\end{align*}
By using those relations, we obtain 
\begin{equation*}
f_1f_{12'2''} = f_{12'2''}f_1 - (q - q\iv)f_{12'}f_{12''}, 
          \quad(\text{similarly for } f_1f_{122''}, f_1f_{122'}). 
\end{equation*}
Also we can compute
\begin{align*}
f_1(f_2f_{2'}f_{2''}) &= f_{122'2''}  
       +  q(f_{2''}f_{122'} + f_{2'}f_{122''} + f_2f_{12'2''})   \\
       &\phantom{*****} + q^2(f_{2'}f_{2''}f_{12} + f_2f_{2''}f_{12'} + f_2f_{2'}f_{12''}) 
       + q^3f_2f_{2'}f_{2''}f_1.
\end{align*}
It follows that 
\begin{equation*}
\tag{3.5.5}
f_1(f_2f_{2'}f_{2''}) \equiv f_{122'2''} + q^3f_2f_{2'}f_{2''}f_1 \mod J.  
\end{equation*}

By multiplying $f_1$ from the left on both sides of (3.5.5), we have 
\begin{align*}
f_1^2(f_2f_{2'}f_{2''}) &\equiv f_1f_{122'2''} + q^3f_1(f_2f_{2'}f_{2''}f_1) \\
            &\equiv f_1f_{122'2''} + q^3(f_{122'2''} + q^3f_2f_{2'}f_{2''}f_1)f_1 \\
            &= f_1f_{122'2''} + q^3f_{122'2''}f_1 + q^6f_2f_{2'}f_{2''}f^2_1, 
\end{align*}
where we again used (3.5.5) in the second identity. On the other hand, 
we can compute
\begin{align*}
\tag{3.5.6}
f_1f_{122'2''} &= qf_{122'2''}f_1 - q(q - q\iv)
          \{ f_{12'2''}f_{12} + f_{122''}f_{12'} + f_{122'}f_{12''}\} \\ 
         &\phantom{*****} + (q\iv -2q)f_{1122'2''}  \\
               &\equiv  qf_{122'2''}f_1  + (q + q\iv)f_{1122'2''} \mod J.  
\end{align*}

Hence we have
\begin{equation*}
\tag{3.5.7}
f_1^2(f_2f_{2'}f_{2''}) \equiv (q + q^3)f_{122'2''}f_1 + (q + q\iv)f_{1122'2''}
                           + q^6f_2f_{2'}f_{2''}f_1^2. 
\end{equation*}

Next by multiplying $f_1$ from the left on both sides of (3.5.7), we have
\begin{equation*}
\tag{3.5.8}
f_1^3(f_2f_{2'}f_{2''}) \equiv (q + q^3)f_1f_{122'2''}f_1 
           + (q + q\iv)f_1f_{1122'2''} + q^6f_1f_2f_{2'}f_{2''}f_1^2.
\end{equation*} 
Here we can compute 
\begin{equation*}
\tag{3.5.9}
f_1f_{1122'2''} = q\iv f_{1122'2''}f_1 + (q - q\iv)^2f_{12}f_{12'}f_{12''}.
\end{equation*}
Thus by applying (3.5.6), (3.5.9) and (3.5.5) to (3.5.8), we have

\begin{align*}
\tag{3.5.10}
f_1^3(f_2f_{2'}f_{2''}) \equiv 
           &(q^6 + q^4 + q^2)f_{122'2''}f_1^2 
          + (q^4 + 2q^2 + 2 + q^{-2})f_{1122'2''}f_1  \\ 
          &+ (q^3 + 2q + 2q\iv + q^{-3})f_{12}f_{12'}f_{12''}
          + q^9f_2f_{2'}f_{2''}f_1^3. 
\end{align*}

Here we note that 
\begin{equation*}
\tag{3.5.11}
f_1f_{12}f_{12'}f_{12''} = q^{-3}f_{12}f_{12'}f_{12''}f_1.
\end{equation*}

Then by multiplying $f_1$ from the left on both sides of (3.5.10), 
and by applying (3.5.6), (3.5.9), (3.5.11) and (3.5.5), we have

\begin{align*}
\tag{3.5.12}
f_1^4(f_2f_{2'}f_{2''}) \equiv 
       (q^9 + &q^7 + q^5 + q^3)f_{122'2''}f_1^3 
      + (q^7 + 2q^5 + 2q\iv + q^{-3})f_{1122'2''}f_1^2  \\ 
      &+ (q^6 + 2q^2 + 2q^{-2} + q^{-6})f_{12}f_{12'}f_{12''}f_1
      + q^{12}f_2f_{2'}f_{2''}f_1^4.
\end{align*}

Now (3.5.1) can be verified easily by (3.5.5), (3.5.7), (3.5.10) and (3.5.12). 
Thus (3.4.2) is verified, and (3.1.3) holds for the case $\ul X$ is of type 
$G_2$. This completes the proof of Proposition 1.10.

\remark{3.6.}
In the case where $\ul X$ is of type $B_2$, the equality (3.3.7) holds 
in $\BU_q^-$. This is also true for the case
of type $G_2$.  In fact, a more precise computation shows that (3.5.1) holds 
in $\BU_q^-$, without appealing modulo $J$ nor modulo 3.

\section{The proof of Lemma 1.13}

\para{4.1.} 
We consider the Cartan matrix as in 3.1.
Since $\ul X$ has rank 2, $\ul w_0$ has two reduced expressions 
$\ul\Bh = (\eta_1, \dots, \eta_{\ul \nu})$ and 
$\ul\Bh' = (\eta_1', \dots, \eta'_{\ul \nu})$. Let $*$ be the anti-algebra 
automorphism  of $\BU_q^-$ and of $\ul\BU_q^-$. 
It is known that 
\begin{align*}
(\ul T_{\eta_1}\cdots \ul T_{\eta_{k-1}}(\ul f_{\eta_k}))^*
   &= \ul T_{\eta'_1}\cdots \ul T_{\eta'_{\ul \nu - k}}(\ul f_{\eta'_{\ul \nu - k +1}}), 
\end{align*}
and the following formula is obtained from the corresponding formula for $\BU_q^-$, 
\begin{equation*}
(R_{\eta_1}\cdots R_{\eta_{k-1}}(\wt f_{\eta_k}))^*
   = R_{\eta'_1}\cdots R_{\eta'_{\ul \nu - k}}(\wt f_{\eta'_{\ul \nu - k + 1}}).
\end{equation*} 
Thus we may verify (1.13.1) for a fixed $\ul\Bh$. 
\par
In the case where $\ul X$ has type $A_1 \times A_1$, there is nothing to prove. 

\para{4.2.}
Assume that $\ul X$ has type $A_2$. 
We write $I = \{ 1, 1', 2, 2' \}$ with $\ul I = \{ \ul 1, \ul 2\}$, 
where $\ul 1 = \{ 1,1'\}, \ul 2 = \{ 2,2'\}$.  
Put $\ul\Bh = (\ul 2, \ul 1, \ul 2)$.  Then $\ul \vD^+ = \{\ul 2, \ul{12}, \ul 1\}$.
We have 
\begin{equation*}
\ul T_{\ul 2}(\ul f_{\ul 1}) = \ul f_{\ul 1}\ul f_{\ul 2} - q^2\ul f_{\ul 2}\ul f_{\ul 1}, 
\quad \ul T_{\ul 2}\ul T_{\ul 1}(\ul f_{\ul 2}) = \ul f_{\ul 1}.
\end{equation*} 
We have 
\begin{align*}
\tag{4.2.1}
R_{\ul 2}(\wt f_{\ul 1}) &= T_2T_{2'}(f_1f_{1'}) = T_2(f_1)T_{2'}(f_{1'}) \\
                         &= (f_1f_2 - qf_2f_1)(f_{1'}f_{2'} - qf_{2'}f_{1'})  \\
      &= f_{1}f_{1'}f_2f_{2'} + q^2f_2f_{2'}f_1f_{1'} - qZ, 
\end{align*}
with
\begin{align*}
Z &= f_1f_2f_{2'}f_{1'} + f_2f_1f_{1'}f_{2'}  \\
&= (f_{12} + qf_2f_1)f_{2'}f_{1'} + f_2f_1(f_{1'2'} + qf_{2'}f_{'}) \\
  &= f_{12}f_{2'}f_{1'} + f_{1'2'}f_2f_1 + 2qf_2f_1f_{2'}f_{1'},
\end{align*}
where $f_{12} = T_2(f_1) = f_1f_2 - qf_2f_1$ 
and $f_{1'2'} = T_{2'}(f_{1'}) = f_{1'}f_{2'}-qf_{2'}f_{1'}$. 
Since $\s(f_{12}) = f_{1'2'}$, we see that $Z \in J$. 
Thus $\pi(R_{\ul 2}(\wt f_1)) = g_{\ul 1}g_{\ul 2} - q^2g_{\ul 2}g_{\ul 1}$
and (1.13.1) holds for $\ul T_2(\ul f_{\ul 1})$. 
Moreover, 
\begin{align*}
\tag{4.2.2}
R_{\ul 2}R_{\ul 1}(\wt f_{\ul 2}) &= T_2T_{2'}T_1T_{1'}(f_2f_{2'}) \\
            &= T_2T_1(f_2)T_{2'}T_{1'}(f_{2'})  \\
            &= f_1f_{1'}. 
\end{align*}
Hence $\pi(R_{\ul 2}R_{\ul 1}(\wt f_{\ul 2})) = g_1$, and (1.13.1) holds for
$\ul T_{\ul 2}\ul T_{\ul 1}(\ul f_{\ul 2})$.  The lemma holds for $\ul X$ of type $A_2$.  

\para{4.3.}
Next assume that $\ul X$ has type $B_2$, and $X$ has type $A_3$. 
We write $I = \{ 2,1,2'\}$ and $\ul I = \{ \ul 1, \ul 2\}$, where 
$\ul 1 = \{ 1 \}, \ul 2 = \{ 2, 2' \}$. Put 
$\Bh = (2,2',1, 2,2',1)$ and $\vD^+ = \{ 2,2',122',12',12,1 \}$. 
Then $\ul\Bh = (\ul 2, \ul 1, \ul 2, \ul 1)$ and 
$\ul\vD^+ = \{ \ul 2, \ul{12}, \ul{112}, \ul 1\}$.
We define root vectors and PBW-bases of $\BU_q^-$ and $\ul\BU_q^-$ similarly 
to the case of $G_2$ in 3.5. 
Then we have
\begin{align*}
\tag{4.3.1}
\ul f_{\ul{12}} &= \ul T_{\ul 2}(\ul f_{\ul 1}) 
              = \ul f_{\ul 1}\ul f_{\ul 2} - q^2\ul f_{\ul 2}\ul f_{\ul 1},  \\
\ul f_{\ul{112}} &= \ul T_{\ul 2}\ul T_{\ul 1}(\ul f_{\ul 2}) 
= (q + q\iv)\iv (\ul f_{\ul 1} \ul f_{\ul{12}} - \ul f_{\ul{12}} \ul f_{\ul 1}), \\ 
\ul f_{\ul 1} &= \ul T_{\ul 2}\ul T_{\ul 1}\ul T_{\ul 2}(\ul f_{\ul 1}).
\end{align*}

We compute 
\begin{align*}
\tag{4.3.2}
R_{\ul 2}(\wt f_{\ul 1}) &= T_2T_{2'}(f_1) = T_2(f_1f_{2'} - qf_{2'}f_1) \\
           &= (f_1f_2 - qf_2f_1)f_{2'} - qf_{2'}(f_1f_2 - qf_2f_1)  \\
           &= f_1f_2f_{2'} + q^2 f_2f_{2'}f_1 
                  - q(f_2f_1f_{2'} + f_{2'}f_1f_2).  
\end{align*}
Hence $\pi(R_{\ul 2}(\wt f_{\ul 1})) = g_{\ul 1}g_{\ul 2} - q^2 g_{\ul 21}g_{\ul 1}$
and (1.13.1) holds for $\ul f_{12}$.  Also 
\begin{align*}
\tag{4.3.3}
R_{\ul 2}R_{\ul 1}R_{\ul 2}(\wt f_{\ul 1}) &= T_2(T_{2'}T_1T_{2'})T_{2}(f_1) \\
             &= T_2(T_1T_{2'}T_1)T_2(f_1) \\
             &= T_2T_1T_{2'}(f_2)  \\
             &= T_2T_1(f_2) = f_1.
\end{align*}
Hence $\pi(R_{\ul 2}R_{\ul 1}R_{\ul 2}(\wt f_{\ul 1})) = g_{\ul 1}$, and 
(1.13.1) holds for $\ul f_{\ul 1}$.
\par
Finally consider 
\begin{align*}
\tag{4.3.4}
R_{\ul 2}R_{\ul 1}(\wt f_{\ul 2}) &=  T_2T_{2'}T_1(f_2f_{2'}) \\
            &= T_{2'}(T_2T_1(f_2))\cdot T_2(T_{2'}T_1(f_{2'})) \\
            &= T_{2'}(f_1)T_2(f_1)  \\
            &= f_{12'}f_{12}.
\end{align*} 

Put 
\begin{align*}
\tag{4.3.5}
Z_{\ul{112}} &= f_1(f_1f_2f_{2'} - q^2f_2f_{2'}f_1) - (f_1f_2f_{2'} - q^2f_2f_{2'}f_1)f_1 \\
  &= f_1^2f_2f_{2'} - (q^2 + 1)f_1f_2f_{2'}f_1  + q^2f_2f_{2'}f_1^2. 
\end{align*}
Clearly $Z_{\ul{112}} \in \BU_q^{-,\s}$, and 
$\pi(Z_{\ul{112}}) = (q + q\iv) \Phi(\ul f_{\ul{112}})$ by (4.3.1).  
We express $Z_{\ul{112}}$ in terms of PBW-basis of $\BU_q^-$. 
By using (3.2.3) for $f_1^2f_2$ and $f_1^2f_{2'}$, we have
\begin{equation*}
\tag{4.3.6}
f_1^2f_2f_{2'} = (q + q\iv)f_1f_2f_1f_{2'} - (q + q\iv)f_2f_1f_{2'}f_1 + f_2f_{2'}f_1^2.
\end{equation*}
\begin{align*}
f_1f_2f_1f_{2'} &= (qf_2f_1 + f_{12})(qf_{2'}f_1 + f_{12'}) \\
       &= q^2f_2f_1f_{2'}f_1 + qf_{12}f_{2'}f_1 + qf_2f_1f_{12'}  + f_{12}f_{12'} \\
       &= q^3f_2f_{2'}f_1^2 + qf_{122'}f_1 + f_2f_{12'}f_1 + f_{12}f_{12'}
            + q^2(f_2f_{12'}f_1 + f_{2'}f_{12}f_1), \\
f_2f_1f_{2'}f_1 &= qf_2f_{2'}f_1^2 + f_2f_{12'}f_1. 
\end{align*}
Here we have used the formula $f_1f_{12'} = q\iv f_{12'}f_1$. 
Moreover, by using $f_{12}f_{2'} = qf_{2'}f_{12} + f_{122'}$, we have
\begin{align*}
\tag{4.3.7}
f_1f_2f_{2'}f_1 &= q^2f_2f_{2'}f_1^2 + q(f_2f_{12'}f_1 + f_{2'}f_{12}f_1) + f_{122'}f_1.
\end{align*}
Substituting these formulas into (4.3.5), we see that 
\begin{equation*}
\tag{4.3.8}
Z_{\ul{112}}  = (q + q\iv)f_{12}f_{12'}.
\end{equation*}
Combining this with (4.3.4), (4.3.5),  
we obtain $\Phi(\ul f_{\ul{112}}) = \pi(R_{\ul 2}R_{\ul 1})(f_1)$.
Thus the lemma holds for $\ul X$ of type $B_2$.  

\para{4.4.}
Finally assume that $\ul X$ has type $G_2$. We follow the notation in 3.5.
Put $\ul\Bh = (\ul 2, \ul 1, \ul 2, \ul 1, \ul 2, \ul 1)$.  Then 
$\ul\vD^+ = \{\ul 2, \ul{12}, \ul{11122}, \ul{112}, \ul{1112}, \ul 1  \}$.
We have
\begin{align*}
\tag{4.4.1}
\ul f_{\ul{12}} &= \ul T_{\ul 2}(\ul f_{\ul 1}) 
                = \ul f_{\ul 1}\ul f_{\ul 2} - q^3\ul f_{\ul 2}\ul f_{\ul 1}, \\
\ul f_{\ul{11122}} &= \ul T_{\ul 2}\ul T_{\ul 1}(\ul f_{\ul 2})
        = [3]_1\iv(\ul f_{\ul{112}}\ul f_{\ul{12}} - q\iv \ul f_{\ul{12}}\ul f_{\ul{112}}), \\  
\ul f_{\ul{112}} &= \ul T_{\ul 2}\ul T_{\ul 1}\ul T_{\ul 2}(\ul f_{\ul 1})
        = [2]_1\iv(\ul f_{\ul 1}\ul f_{\ul{12}} - q\ul f_{\ul{12}}\ul f_{\ul 1}), \\
\ul f_{\ul{1112}} &= \ul T_{\ul 2}\ul T_{\ul 1}\ul T_{\ul 2}\ul T_{\ul 1}(\ul f_{\ul 2})
       = [3]_1\iv(\ul f_{\ul 1}\ul f_{\ul{112}} - q\iv \ul f_{\ul{112}}\ul f_{\ul 1})  \\
\ul f_{\ul 1} &= \ul T_{\ul 2}\ul T_{\ul 1}\ul T_{\ul 2}\ul T_{\ul 1}\ul T_{\ul 2}(\ul f_{\ul 1}). 
\end{align*} 

First consider the case $\ul f_{\ul{12}}$. By using (4.3.2), we have
\begin{align*}
R_{\ul 2}(\wt f_{\ul 1}) &= T_2T_{2'}T_{2''}(f_1) \\
       &= T_{2''}\bigl(f_1f_2f_{2'} + q^2f_2f_{2'}f_1 
                  - q(f_2f_1f_{2'} + f_{2'}f_1f_2) \bigr) \\
       &= f_1(f_2f_{2'}f_{2''}) - q^3(f_2f_{2'}f_{2''})f_1 \\
       &\phantom{***} -q(f_2f_1f_{2'}f_{2''} + f_{2'}f_1f_{2''}f_2 + f_{2''}f_1f_2f_{2'})  \\
       & \phantom{***}+q^2(f_2f_{2'}f_1f_{2''} + f_{2'}f_{2''}f_1f_2 + f_{2''}f_2f_1f_{2'}).
\end{align*}
Hence $\pi(R_{\ul 2}(\wt f_1)) = g_{\ul 1}g_{\ul 2} - q^3g_{\ul 2}g_{\ul 1}$
and (1.13.1) holds for $\ul f_{\ul{12}}$. 

\par
Next consider the case $\ul f_{\ul{112}}$. Put
\begin{align*}
Z_{\ul{12}}  &= f_1(f_2f_{2'}f_{2''}) - q^3(f_2f_{2'}f_{2''})f_1, \\
Z_{\ul{112}} &= f_1Z_{\ul {12}} - qZ_{\ul{12}}f_1.
\end{align*}
Then we have
\begin{align*}
\tag{4.4.2}
Z_{\ul{112}} = 
   f_1^2f_2f_{2'}f_{2''} - (q^3 + q)f_1f_2f_{2'}f_{2''}f_1 + q^4f_2f_{2'}f_{2''}f_1^2. 
\end{align*}
Clearly $Z_{\ul{112}} \in \BU_q^{-,\s}$, and we have 
\begin{equation*}
\tag{4.4.3}
\pi(Z_{\ul{112}}) = (q + q\iv)\Phi(\ul f_{\ul{112}})
\end{equation*}
by (4.4.1). We express each term of $Z_{\ul{112}}$ in terms of PBW-basis. 
By (3.5.5), we have 
\begin{align*}
\tag{4.4.4}
f_1(f_2f_{2'}f_{2''})f_1 \equiv f_{122'2''}f_1 + q^3f_2f_{2'}f_{2''}f_1^2 
        \quad \mod J.  
\end{align*}
By (3.5.7), we have
\begin{align*}
\tag{4.4.5}
f_1^2(f_2f_{2'}f_{2''}) \equiv (q &+ q^3)f_{122'2''}f_1 + (q + q\iv)f_{1122'2''}
             + q^6f_2f_{2'}f_{2''}f_1^2 \quad \mod J. 
 \end{align*}
Substituting these formulas into (4.4.2), we have 
$Z_{\ul{112}} \equiv (q + q\iv)f_{1122'2''} \mod J$, which implies that
\begin{equation*}
\tag{4.4.6}
\pi(Z_{\ul{112}}) = (q + q\iv)\pi(f_{1122'2''}).
\end{equation*}
Note that by (3.5.2) and (3.5.3), we have 
\begin{equation*}
R_{\ul 2}R_{\ul 1}R_{\ul 2}(f_1) = T_2T_{2'}T_{2''}T_1T_2T_{2'}T_{2''}(f_1) = f_{1122'2''}. 
\end{equation*}
By comparing (4.4.3) and (4.4.6), we obtain 
\begin{equation*}
\tag{4.4.7}
\pi(R_{\ul 2}R_{\ul 1}R_{\ul 2}(f_1)) = \Phi(\ul f_{\ul{112}}).
\end{equation*}
Thus (1.13.1) holds for $\ul f_{\ul{112}}$. 
\par
Next consider the case of $\ul f_{\ul{1112}}$. 
Put 
\begin{equation*}
Z_{\ul{1112}} = f_1Z_{\ul{112}} - q\iv Z_{\ul{112}}f_1.
\end{equation*}
It follows from the computation of $Z_{\ul{112}}$ in (4.4.2), we have
\begin{align*}
\tag{4.4.8}
Z_{\ul{1112}} = f_1^3&f_2f_{2'}f_{2''} - (q^3 + q + q\iv)f_1^2f_2f_{2'}f_{2''}f_1  \\
                &+ (q^4 + q^2 + 1)f_1f_2f_{2'}f_{2''}f_1^2 - q^3f_2f_{2'}f_{2''}f_1^3.
\end{align*}

\par
Clearly $Z_{\ul{1112}} \in \BU_q^{-,\s}$, and we have
\begin{equation*}
\tag{4.4.9}
\pi(Z_{\ul{1112}}) = [2]_1[3]_1\Phi(\ul f_{\ul{1112}}). 
\end{equation*}
By (3.5.10), we have
\begin{align*}
f_1^3(f_2f_{2'}f_{2''}) \equiv (q^6 &+ q^4 + q^2)f_{122'2''}f_1^2 
          + (q^4 + 2q^2 + 2 + q^{-2})f_{1122'2''}f_1 \\
         &+ (q^3 + 2q^2 + 2q\iv + q^{-3})f_{12}f_{12'}f_{12''} + q^9f_2f_{2'}f_{2''}f_1^3.
\end{align*}
By this formula together with (3.5.7) and (3.5.5), we have
$Z_{\ul{1112}} = [2]_1[3]_1f_{12}f_{12'}f_{12''} \mod J$, which implies that. 
\begin{align*}
\tag{4.4.10}
\pi(Z_{\ul{1112}}) = [2]_1[3]_1\pi(f_{12}f_{12'}f_{12''}).
\end{align*}
Note that by (3.5.2) and (3.5.3), we have
\begin{align*}
R_{\ul 2}R_{\ul 1}R_{\ul 2}R_{\ul 1}(\wt f_{\ul 2})
      &= T_2T_{2'}T_{2''}T_1T_2T_{2'}T_{2''}T_1(f_2f_{2'}f_{2''})  \\
      &= f_{12}f_{12'}f_{12''}.
\end{align*}
By comparing (4.4.9) and (4.4.10), we obtain 
\begin{equation*}
\pi(R_{\ul 2}R_{\ul 1}R_{\ul 2}R_{\ul 1}(\wt f_{\ul 2})) = \Phi(f_{\ul{1112}}).
\end{equation*}
Thus (1.13.2) holds for $\ul f_{\ul{1112}}$. 
\par
Finally consider the case of $\ul f_{\ul{11122}}$. 
Put
\begin{equation*}
Z_{\ul{11122}} = f_{1122'2''}f_{122'2''} - q\iv f_{122'2''}f_{1122'2''}.
\end{equation*}
By (3.5.2) and (3.5.3), we have
\begin{align*}
R_{\ul 2}(f_1) = T_2T_{2'}T_{2''}(f_1) = f_{122'2''}. 
\end{align*}
Hence, by the previous computation, 
we know that $\pi(f_{122'2''}) = \Phi(\ul f_{\ul{12}})$.  
On the other hand, by (4.4.7), we have $\pi(f_{1122'2''}) = \Phi(\ul f_{\ul{112}})$.
It follows, by (4.4.1), that 
\begin{align*}
\tag{4.4.11}
\pi(Z_{\ul{11122}}) = [3]_1\Phi(\ul f_{\ul{11122}}).
\end{align*}
We note, by (3.5.2) and (3.5.3),  that
\begin{align*}
R_{\ul 2}R_{\ul 1}(\ul f_{\ul 2}) = T_2T_{2'}T_{2''}T_1(f_2f_{2'}f_{2''}) 
                                  = f_{12'2''}f_{122''}f_{122'}. 
\end{align*} 
Thus in order to prove (1.13.1) for $\ul f_{\ul{11122}}$, 
it is enough to see that 
\begin{equation*}
\tag{4.4.12}
Z_{\ul{11122}} \equiv [3]_1f_{12'2''}f_{122''}f_{122'} \mod J.
\end{equation*}

We shall express $Z_{\ul{11122}}$  in terms of the PBW-basis of $\BU_q^-$.  
In the computation below, in addition to the formulas in 3.5, 
we need to use the following commutation relations, 
which are deduced from  the formula of Levendorskii and Soibelman [LS] 
applied for the subalgebra of type $A_3$. .

\begin{align*}
\tag{4.4.13}
f_{12}f_2 &= q\iv f_2f_{12},  \\
f_{122''}f_2 &= q\iv f_2f_{122''}, \\
f_{122'}f_2 &= q\iv f_2f_{122'}, \\
f_{12}f_{122'} &= q\iv f_{122'}f_{12}, \\
f_{12}f_{122''} &= q\iv f_{122''}f_{12},
\end{align*}
and the formulas (two for each) by applying the operation $\s$ on
both sides. By using these relations, we have
\begin{align*}
\tag{4.4.14}
f_{12}f_{12'2''} &= f_{122'2''}f_{12} + (q\iv - q)f_{122''}f_{122'}, \\
f_{1122'2''}f_2 &= f_2f_{1122'2''} + (q\iv -q)f_{122''}f_{122'}, 
\end{align*}
and the formulas (two for each) by applying the operation $\s$ on 
both sides. 
\par
Now we can compute (note that the second formula in (4.4.14) is not
used in this computation) 
\begin{align*}
f_{1122'2''}f_{122'2''} = (q^2  - 2 + q^{-2})f_{12'2''}f_{122''}f_{122'} 
                     + q\iv f_{122'2''}f_{1122'2''}.
\end{align*}
Hence
\begin{align*}
Z_{\ul{11122}} &= f_{1122'2''}f_{122'2''} - q\iv f_{122'2''}f_{1122'2''} \\
               &= (q^2 - 2 + q^{-2})f_{12'2''}f_{122''}f_{122'} \\
               &\equiv [3]_1f_{12'2''}f_{122''}f_{122'} \mod J. 
\end{align*}
Thus (4.4.12) holds, and (1.13.1) is proved for $\ul f_{\ul{11122}}$. 
The lemma holds for $\ul X$ of type $G_2$. 

\par

\par\vspace{1.5cm}
\noindent
T. Shoji \\
School of Mathematical Sciences, Tongji University \\ 
1239 Siping Road, Shanghai 200092, P.R. China  \\
E-mail: \verb|shoji@tongji.edu.cn|

\par\vspace{0.5cm}
\noindent
Z. Zhou \\
School of Mathematical Sciences, Tongji University \\ 
1239 Siping Road, Shanghai 200092, P.R. China  \\
E-mail: \verb|forza2p2h0u@163.com|


\end{document}